\newif\ifsiam\siamfalse

\PassOptionsToPackage{usenames,dvipsnames,table}{xcolor}
\ifsiam
	\documentclass[review,onefignum,onetabnum]{siamart190516}
\else
	\documentclass[USenglish,10pt,a4paper,reqno]{article}
\fi

	\PassOptionsToPackage{capitalize,nameinlink}{cleveref}
	\PassOptionsToPackage{breakspaces}{clevethm}
	\PassOptionsToPackage{dvipsnames}{xcolor}

	\usepackage[cal=boondoxo]{mathalfa}
	\ifsiam
		\usepackage[
			noload={amsthm,txfonts,tikz,pgfplots,caption,subcaption},
			excludethm=all,
			customsettings=false,
		]{myPreamble}
	\else
		\usepackage[
			noqed,
			noload={tikz,pgfplots},
			eqreset=section,
			thmreset=section,
		]{myPreamble}
	\fi
	\let\OLDand\and
	\def\and{\texorpdfstring{\OLDand}{, }}%
	\ifsiam
		\let\qedhere\relax
		\newcommand{\Sep}{,\ }
		\renewenvironment{qedequation*}{\[}{\]}
	\else
		\newenvironment{keywords}{\par\noindent{\bf Keywords. }}{}
		\newenvironment{AMS}{\par\noindent{\bf AMS subject classifications. }}{}
		\newcommand{\Sep}{ \(\cdot\)\ }
	\fi
	\crefformat{equation}{(#2#1#3)}
	\renewcommand{\proofitemizelabeli}{}

	\ifsiam
		\Aligntrue
		\let\thedummythm\thetheorem

		\newsiamremark{remark}{Remark}
		\newsiamremark{example}{Example}

		\theoremstyle{plain}
		\theoremheaderfont{\normalfont\itshape}
		\theorembodyfont{\normalfont}
		\theoremseparator{.}
		\theoremsymbol{}
		\newtheorem{assumption}{Assumption}
		
		\Crefname{assumption}{Assumption}{Assumptions}

		\CreateTheoremLists{corollary}
		\CreateTheoremLists{definition}
		\CreateTheoremLists{example}
		\CreateTheoremLists{fact}
		\CreateTheoremLists{lemma}
		\CreateTheoremLists{proposition}
		\CreateTheoremLists{remark}
		\CreateTheoremLists{theorem}
		
		\makeatletter
			\let\oldproposition\proposition
			\let\oldendproposition\endproposition
			\renewenvironment{proposition}[1][]{%
				\def\@currentlabelname{#1}\ifstrempty{#1}{\oldproposition}{\oldproposition[#1]}%
			}{\oldendproposition}
			\let\oldlemma\lemma
			\let\oldendlemma\endlemma
			\renewenvironment{lemma}[1][]{%
				\def\@currentlabelname{#1}\ifstrempty{#1}{\oldlemma}{\oldlemma[#1]}%
			}{\oldendlemma}
			\let\oldtheorem\theorem
			\let\oldendtheorem\endtheorem
			\renewenvironment{theorem}[1][]{%
				\def\@currentlabelname{#1}\ifstrempty{#1}{\oldtheorem}{\oldtheorem[#1]}%
			}{\oldendtheorem}
		\makeatother

		\foreach \x/\y in {corollary/cor, definition/defin, example/es, fact/fact, lemma/lem, proposition/prop, remark/rem, theorem/thm} {
			\RegisterTheoremName{\x}{\y}
		}

		\setlist*[proofitemize,1]{itemindent=\parindent}
	\else
		\Alignfalse
		\newenvironment{assumption}{\begin{ass}}{\end{ass}}
		\newenvironment{corollary}{\begin{cor}}{\end{cor}}
		\newenvironment{definition}{\begin{defin}}{\end{defin}}
		\newenvironment{example}{\begin{es}}{\end{es}}
		\newenvironment{lemma}{\begin{lem}}{\end{lem}}
		\newenvironment{proposition}{\begin{prop}}{\end{prop}}
		\newenvironment{remark}{\begin{rem}}{\end{rem}}
		\newenvironment{theorem}{\begin{thm}}{\end{thm}}
	\fi

	\makeatletter
		\newcommand{\BiFRB}{%
			\hyperref[alg:BiFRB]{\(i^*\)FRB}%
			\@ifstar{\ (\cref{alg:BiFRB})}{}%
		}%
	\makeatother

	\let\closure\overline
	\newcommand{\myhat}[1]{\vphantom{#1}\smash{\hat{#1}}}
	\renewcommand{\nicefrac}[2]{#1/#2}
	\renewcommand{\j}{\mathcal j}

	\newcommand{\stepsize}{\gamma}
	\newcommand{\inertia}{\beta}
	\newcommand{\kernel}{h}

	\newcommand{\C}{\mathcal C} 
	\newcommand{\cost}{\varphi_{\closure C}}
	\newcommand{\envelope}{\phi}
	\newcommand{\lyapunov}{\mathcal L}
	\newcommand{\model}{\mathcal M}
	\newcommand{\operator}{\operatorname T}

	\newcommand{\tildeall}{%
		\let\oldenvelope\envelope
		\let\oldlyapunov\lyapunov
		\let\oldmodel\model
		\let\oldoperator\operator
		\let\oldstepsize\stepsize
		\let\oldinertia\inertia
		\renewcommand{\envelope}{\tilde{\oldenvelope}}%
		\renewcommand{\lyapunov}{\tilde{\oldlyapunov}}%
		\renewcommand{\model}{\tilde{\oldmodel}}%
		\renewcommand{\operator}{\tilde{\oldoperator}}%
		\renewcommand{\stepsize}{\tilde{\oldstepsize}}%
		\renewcommand{\inertia}{\tilde{\oldinertia}}%
	}%
	\newcommand{\h}{\hat{\kernel}}
	\newcommand{\f}{\myhat{f}_{\!\inertia}}

	\newcommand{\env}{\envelope_{\stepsize\!,\,\inertia}^{\kernel\text{-\sc frb}}}
	\newcommand{\LL}{\lyapunov_{\stepsize\!,\,\inertia}^{\kernel\text{-\sc frb}}}
	\newcommand{\M}{\model_{\stepsize\!,\,\inertia}^{\kernel\text{-\sc frb}}}
	\newcommand{\T}{\operator_{\stepsize\!,\,\inertia}^{\kernel\text{-\sc frb}}}
	\makeatletter
		\newcommand{\D}{\operatorname D\@D}
		\def\@D{\@ifnextchar_{\@D@subarg}{\@D@subarg_{\kernel}}}
		\def\@D@subarg_#1{_{#1}}
	\makeatother

\crefname{ALG@line}{step}{steps}

	\SetEnumitemKey{B}{%
		B*,
		topsep = 0pt,
		before = \removelastskip,
	}
	\SetEnumitemKey{B*}{%
		label={\it(B$_{\arabic*}$)},
		ref={\thedummythm{\it(B$_{\arabic*}$)}},
	}

\newcommand{\F}{\mathcal F_{\gamma,\beta}^{h\text{\sc-frb}}}
\newcommand{\vertiii}[1]{{
	\vert\kern-0.2ex\vert\kern-0.2ex\vert #1
    \vert\kern-0.2ex\vert\kern-0.2ex\vert
}}
\makeatletter
	\newcommand{\norm}{\@ifstar\@@norm\@norm}
	\newcommand{\@norm}[1]{\|{}#1{}\|}
	\newcommand{\@@norm}[1]{\vertiii{#1}}
\makeatother
\newcommand{\inte}{\ensuremath{\operatorname{int}}}

\newcommand{\TheKeywords}{%
   Nonsmooth nonconvex optimization\Sep
   forward-reflected-backward splitting\Sep
   inertia\Sep
   Bregman distance\Sep
   relative smoothness.%
}
\newcommand{\TheAMSsubj}{%
	90C26\Sep 
	49J52\Sep 
	49J53.
}
\newcommand{\TheFunding}{%
	This work was supported by NSERC Discovery Grants,
	and the JSPS KAKENHI grant number JP21K17710.%
}

\author{%
	Ziyuan Wang\thanks{%
		Department of Mathematics, Irving K. Barber Faculty of Science, University of British Columbia,
		Kelowna, B.C. V1V 1V7, Canada.%
		\texorpdfstring{\newline}{ }%
		{\it E-mails:}
		{\sf
			\href{mailto:shawn.wang@ubc.ca}{shawn.wang@ubc.ca},
			\href{mailto:ziyuan.wang@alumni.ubc.ca}{ziyuan.wang@alumni.ubc.ca}%
		}%
	}%
	\and
	Andreas Themelis\thanks{%
		Kyushu University,
		Faculty of Information Science and Electrical Engineering (ISEE),
		744 Motooka, Nishi-ku, 819-0395 Fukuoka, Japan.%
		\texorpdfstring{\newline}{ }%
		{\it E-mails:}
		{\sf
			\href{mailto:ou.honjia.069@s.kyushu-u.ac.jp}{ou.honjia.069@s.kyushu-u.ac.jp},
			\href{mailto:andreas.themelis@ees.kyushu-u.ac.jp}{andreas.themelis@ees.kyushu-u.ac.jp}%
		}%
	}%
	\and
	Hongjia Ou\footnotemark[3]%
	\and
	Xianfu Wang\footnotemark[2]%
}

\ifsiam
	\headers{A mirror inertial forward-reflected-backward splitting}{Z. Wang, A. Themelis, H. Ou, and X. Wang}
	\title{%
		A mirror inertial forward-reflected-backward splitting:
		Global convergence and linesearch extension beyond convexity and Lipschitz smoothness%
		\thanks{%
			Submitted to the editors \today.%
			\funding{\TheFunding}%
		}%
	}
\else
	\date{}
	\title{\Large%
		A mirror inertial forward-reflected-backward splitting:
		Global convergence and linesearch extension beyond convexity and Lipschitz smoothness%
		\thanks{\TheFunding}%
	}
\fi

\showgrayfalse
\begin{document}
	
	\maketitle
	\begin{abstract}
		This work investigates a Bregman and inertial extension of the forward-re\-flect\-ed-backward algorithm
[Y. Malitsky and M. Tam, \emph{SIAM J. Optim.}, 30 (2020), pp. 1451--1472]
applied to structured nonconvex minimization problems under relative smoothness.
To this end, the proposed algorithm hinges on two key features: taking inertial steps in the \emph{dual} space, and allowing for \emph{possibly negative} inertial values.
Our analysis begins with studying an associated envelope function that takes inertial terms into account through a novel product space formulation.
Such construction substantially differs from similar objects in the literature and could offer new insights for extensions of splitting algorithms.
Global convergence and rates are obtained by appealing to the generalized concave Kurdyka-\L ojasiewicz~(KL) property, which allows us to describe a sharp upper bound on the total length of iterates.
Finally, a linesearch extension is given to enhance the proposed method.

	\end{abstract}
	

	\section{Introduction}
		Consider the following composite minimization problem
\begin{equation*}
	\tag{P}\label{eq:P}
	\minimize_{x\in\R^n}\varphi(x)\coloneqq f(x)+g(x)
\quad
	\stt x\in\closure C,
\end{equation*}
where \(C\) is a nonempty open and convex set with closure \(\closure C\),  \(\func{f}{\R^n}{\Rinf\coloneqq\R\cup\set{\pm\infty}}\) is differentiable on $\inte\dom f\neq\emptyset$, and \(\func{g}{\R^n}{\Rinf}\) is proper and lower semicontinuous (lsc).
For notational brevity, we define \(\cost\coloneqq\varphi+\indicator_{\closure C}\) with \(\indicator_X\) denoting the indicator function of set \(X\subseteq\R^n\), namely such that \(\indicator_X(x)=0\) if \(x\in X\) and \(\infty\) otherwise.
By doing so, problem \eqref{eq:P} can equivalently be cast as the ``unconstrained'' minimization
\[
	\minimize_{x\in\R^n}\cost(x).
\]
Note that~\eqref{eq:P} is beyond the scope of traditional first-order methods that require global Lipschitz continuity of \(\nabla f\) and the consequential descent lemma \cite[Prop. A.24]{bertsekas2016nonlinear};
see, e.g, \cite{attouch2013convergence,parikh2014proximal,mairal2015incremental,li2015global,li2016douglas,li2017peaceman} for such algorithms.
To resolve this issue, Lipschitz-like convexity was introduced in the seminal work \cite{bauschke2017descent}, furnishing a descent lemma beyond the aforementioned setting.
This notion was then referred to as relative smoothness~(see \cref{def:relsmooth}) and has played a central role in extending splitting algorithm to the setting of~\eqref{eq:P}; see, e.g., \cite{bolte2018first,dragomir2021quartic,hanzely2021accelerated,lu2018relatively,nesterov2019implementable,teboulle2018simplified}.

The goal of this paper is to propose a Bregman inertial forward-reflected-backward method \BiFRB* for solving~\eqref{eq:P}, which, roughly speaking, iterates
\[
	x^{k+1}
{}\in{}
	(\nabla h+\gamma\partial g)^{-1}(\nabla h(x^k)+\beta(\nabla h(x^k)-\nabla h(x^{k-1})-\gamma(2\nabla f(x^k)-\nabla f(x^{k-1})),
\]
where $\gamma>0$ is the stepsize, $\beta$ is inertial parameter, and $h$ is the kernel.
In the convex case, the above scheme reduces to the inertial forward-reflected-backward method proposed in~\cite{malitsky2020forward} when $h=(\nicefrac12)\|\cdot\|^2$, which is not applicable to~\eqref{eq:P} due to its assumption on Lipschitz continuity of $\nabla f$.

A fundamental tool in our analysis is the \BiFRB-envelope~(see \cref{def: envelope}), which is the value function associated to the parametric minimization of a ``model'' of~\eqref{eq:P}; see \cref{sec:model}.
The term ``envelope'' is borrowed from the celebrated Moreau envelope~\cite{moreau1965proximite} and its relation with the proximal operator.
Indeed, there has been a re-emerged interest of employing an associated envelope function to study convergence of splitting methods, such as forward-backward splitting \cite{themelis2018forward,ahookhosh2021bregman}, Douglas-Rachford splitting \cite{themelis2020douglas,themelis2022douglas}, alternating minimization algorithm \cite{stella2018newton}, as well as the splitting scheme of Davis and Yin \cite{liu2019envelope}.
The aforementioned works share one common theme: regularity properties of the associated envelope function are used for further enhancement and deeper algorithmic insights.

Pursuing the same pattern, additionally to studying \emph{global convergence} of the proposed algorithm in \cref{sec:convergence} using the \BiFRB-envelope, we will showcase in \cref{sec:CLyD} how it offers a \emph{versatile globalization framework} for fast local methods in the full generality of \cref{ass:basic}.
Such framework revealed the necessity of interleaving noninertial trial steps in the globalization strategy, an observation that led to a substantial change in the linesearch strategy with respect to related works.

A \emph{major departure} of our analysis from previous work is that we consider an envelope function with two independent variables, allowing us to take inertial terms into account. In this regard, we believe that our methodology is appealing in its own right, as it can be instrumental for deriving inertial extensions of other splitting methods.
Another notable feature of this work is that, as one shall see in \cref{sec:bounds}, \emph{non-positive inertial parameter} is required for the sake of convergence under relative smoothness.
This result, although more pessimistic, aligns with the recent work~\cite{dragomir2022optimal} regarding the impossibility of accelerated Bregman forward-backward method under the same assumption; see \cref{rem:negative inertia} for a detailed discussion.
Our work \emph{differs} from the analysis carried out in \cite{wang2022bregman}, which also deals with an inertial forward-reflected-backward algorithm using Bregman metrics but is still limited to the Lipschitz smoothness assumption.
The game changer that enables us to cope with the relative smoothness is taking the inertial step in the \emph{dual space}, that is, interpolating application of \(\nabla h\) (cf. \cref{state:BiFRB:x} of \cref{alg:BiFRB}), whence the name, inspired by \cite{beck2003mirror}, \emph{mirror inertial} forward-reflected-backward splitting (\BiFRB).

	The rest of the paper is structured as follows.
	In the remainder of the section we formally define the problem setting and the proposed algorithm, and in \cref{sec:preliminaries} we discuss some preliminary material and notational conventions.
	\Cref{sec:toolbox} introduces the \BiFRB-envelope and an associated merit function; the proof of the main result therein is deferred to \cref{sec:SD:proof}.
	The convergence analysis is carried out in \cref{sec:convergence}, and finally \cref{sec:CLyD} presents the \BiFRB-based globalization framework.
	\Cref{sec:conclusions} draws some concluding remarks.

		\subsection{Problem setting and proposed algorithm}
			Throughout, we fix a Legendre kernel \(\func{h}{\R^n}{\Rinf}\) with \(\dom\nabla h=\interior\dom h=C\),
namely, a proper, lsc, and strictly convex function  that is 1-coercive and \DEF{essentially smooth}, \ie such that \(\|\nabla h(x_k)\|\to\infty\) for every sequence \(\seq{x_k}\subset C\) converging to a boundary point of \(C\).
We will consider the following iterative scheme for addressing problem \eqref{eq:P},
where \(\func{\D}{\R^n\times\R^n}{\Rinf}\) denotes the \DEF{Bregman distance} induced by \(h\), defined as
\begin{equation}\label{eq:D}
	\D(x,y)
{}\coloneqq{}
	\begin{ifcases}
		h(x)-h(y)-\innprod{\nabla h(y)}{x-y} & y\in C,
	\\
		\infty\otherwise.
	\end{ifcases}
\end{equation}

\begin{algorithm}
	\caption{Mirror inertial forward-reflected-backward (\protect\BiFRB)}%
	\label{alg:BiFRB}%
	\begin{algorithmic}[1]%
\setlength{\itemsep}{1ex}%
\item[%
	Choose \(x^{-1},x^0\in C\),~
	inertial parameter \(\beta\in\R\),~ and
	stepsize \(\gamma>0\)%
]%
\item[%
	Iterate for \(k=0,1,\ldots\) until a termination criterion is met (cf. \cref{thm:termination})
]%
\State\label{state:BiFRB:y}%
	Set \(y^k\) such that
	\(
		\nabla h(y^k)
	{}={}
		\nabla h(x^k)-\gamma\bigl(\nabla f(x^k)-\nabla f(x^{k-1})\bigr)
	\)%
\State\label{state:BiFRB:x}%
	Choose
	\(
	\ifsiam
		\mathtight
	\fi
		x^{k+1}
	{}\in{}
		\argmin\limits_{w\in\R^n}\set{
			g(w)
			{}+{}
			\innprod*{w}{\nabla f(x^k)-\tfrac\beta\gamma\bigl(\nabla h(x^k)-\nabla h(x^{k-1})\bigr)}
			{}+{}
			\tfrac1\gamma\D(w,y^k)
		}
	\)%
\end{algorithmic}

\end{algorithm}

Note that \cref{alg:BiFRB} takes inertial step in the dual space, hence the abbreviation \BiFRB.
We will work under the following assumptions.

\begin{assumption}\label{ass:basic}%
	The following hold in problem \eqref{eq:P}:
	\begin{enumeratass}
	\item\label{ass:f}%
		\(\func{f}{\R^n}{\Rinf}\) is smooth relative to \(h\) (see \cref{sec:relsmooth}).
	\item\label{ass:g}%
		\(\func{g}{\R^n}{\Rinf}\) is proper and lsc.
	\item\label{ass:phi}%
		\(\inf\cost>-\infty\).
	\item\label{ass:T}%
		For any \(v\in\R^n\) and \(\gamma>0\), \(\argmin\set{\gamma g+h-\innprod{v}{\cdot}}\subseteq C\).
	\end{enumeratass}
\end{assumption}

As will be made explicit in \cref{thm:basic}, \cref{ass:T} is a requirement ensuring that \cref{alg:BiFRB} is well defined.
Note that, in general, the minimizers therein are a (possibly empty) subset of \(\dom h\cap\dom g\); \cref{ass:T} thus only excludes points on the boundary of \(\dom h\).
This standard requirement is trivially satisfied when \(\dom h\) is open, or more generally when constraint qualifications enabling a subdifferential calculus rule on the boundary are met, as is the case when \(g\) is convex.

\begin{remark}[constraint qualifications for \cref{ass:T}]%
	If \(g\) is proper and lsc, \cref{ass:T} is satisfied if \(\partial^\infty g\cap\bigl(-\partial^\infty h\bigr)\subseteq\set0\) holds everywhere (this condition being automatically guaranteed at all point outside the boundary of \(C\), having \(\partial^\infty h\) empty outside \(\dom h\) and \(\set0\) in its interior).
	Indeed, optimality of \(\bar x\in\argmin\set{\gamma g+h-\innprod{v}{\cdot}}\) implies that
	\(
		v\in\partial[\gamma g+h](\bar x)\subseteq\gamma\partial g(\bar x)+\partial h(\bar x)
	\),
	with inclusion holding by \cite[Cor. 10.9]{rockafellar2011variational} and implying nonemptiness of \(\partial h(\bar x)\); see \cref{sec:notation} for definitions of subdifferentials. 
\end{remark}

Regardless, it will be shown in \cref{thm:proxbounded} that existence of minimizers is guaranteed for small enough values of \(\gamma\), which will then be linked in \cref{thm:basic} to the well definedness of \cref{alg:BiFRB}.

	\section{Preliminaries}\label{sec:preliminaries}
		
		\subsection{Notation}\label{sec:notation}
			The extended-real line is denoted by \(\Rinf\coloneqq\R\cup\set{\pm\infty}\).
The positive and negative part of \(r\in\R\) are respectively defined as \([r]_+\coloneqq\max\set{0,r}\) and \([r]_-\coloneqq\max\set{0,-r}\), so that \(r=[r]_+-[r]_-\).

The distance of a point \(x\in\R^n\) to a nonempty set \(S\subseteq\R^n\) is given by \(\dist(x,S)=\inf_{z\in S} \|z-x\|\).
The interior, closure, and boundary of \(S\) are respectively denoted as \(\interior S\), \(\overline S\), and \(\boundary S=\overline S\setminus\interior S\).
The \DEF{indicator function} of \(S\) is \(\func{\indicator_S}{\R^n}{\Rinf}\) defined as \(\indicator_S(x)=0\) if \(x\in S\) and \(\infty\) otherwise.

A function \(\func{f}{\R^n}{\Rinf}\) is \DEF{proper} if \(f\not\equiv\infty\) and \(f>-\infty\), in which case its \DEF{domain} is defined as the set \(\dom f\coloneqq\set{x\in\R^n}[f(x)<\infty]\).
For \(\alpha\in\R\), \([f\leq\alpha]\coloneqq\set{x\in\R^n}[f(x)\leq\alpha]\) denotes the \(\alpha\)-sublevel set of \(f\);
	\([\alpha\leq f\leq\beta]\) with \(\alpha,\beta\in\R\) is defined accordingly. We say that \(f\) is \DEF{level bounded} (or \DEF{coercive}) if \(\liminf_{\|x\|\to\infty}f(x)=\infty\), and \(1\)-coercive if
\(\lim_{\|x\|\to\infty}f(x)/\|x\|=\infty\).
A point \(x_\star\in\dom f\) is a \DEF{local minimum} for \(f\) if \(f(x)\geq f(x_\star)\) holds for all \(x\) in a neighborhood of \(x_\star\).
If the inequality can be strengthened to \(f(x)\geq f(x_\star)+\tfrac\mu2\|x-x_\star\|^2\) for some \(\mu>0\), then \(x_\star\) is a \DEF{strong local minimum}.
The \DEF{convex conjugate} of \(f\) is denoted as \(\conj f\coloneqq\sup_z\set{\innprod{{}\cdot{}}{z}-f(z)}\).
Given \(x\in\dom f\), \(\partial f(x)\) denotes the \DEF{Mordukhovich (limiting) subdifferential} of \(f\) at \(x\), given by
\[
	\partial f(x)
{}\coloneqq{}
	\set{v\in\R^n}[
		\exists\seq{x^k,v^k}~\text{s.t.}~ x^k\to x,~f(x^k)\to f(x),~
		\hat\partial f(x^k)\ni v^k\to v
	],
\]
and \(\hat\partial f(x)\) is the set of \DEF{regular subgradients} of \(f\) at \(x\), namely vectors \(v\in\R^n\) such that
\(
	\liminf_{\limsubstack{z&\to&x\\z&\neq&x}}{
		\frac{
			f(z)-f(x)-\innprod{v}{z-x}
		}{
			\|z-x\|
		}
	}
{}\geq{}
	0
\);
see, e.g.,~\cite{rockafellar2011variational,mordukhovich2018variational}.
\(\C^k(\mathcal U)\) is the set of functions \(\mathcal U\to\R\) which are \(k\) times continuously differentiable.
We write \(\C^k\) if \(\mathcal U\) is clear from context.
The notation \(\ffunc{T}{\R^n}{\R^n}\) indicates a set-valued mapping, whose \DEF{domain} and \DEF{graph} are respectively defined as
\(
	\dom T=\set{x\in\R^n}[T(x)\neq\emptyset]
\)
and
\(
	\graph T=\set{(x,y)\in\R^n\times\R^n}[y\in T(x)]
\).

		\subsection{Relative smoothness and weak convexity}\label{sec:relsmooth}
			In the following definition, as well as throughout the entire paper, we follow the ex\-tend\-ed-real convention \(\infty-\infty=\infty\) to resolve possible ill definitions of difference of extended-real--valued functions.
Similarly, we will adopt the convention that \(\nicefrac10=\infty\).

\begin{definition}[relative smoothness]\label{def:relsmooth}%
	We say that an lsc function \(\func{f}{\R^n}{\Rinf}\) is smooth relative to \(h\) if \(\dom f\supseteq\dom h\) and there exists a constant \(L_{f,h}\geq0\) such that \(L_{f,h}h\pm f\) are convex functions  on \(\inte\dom h\).
	We may alternatively say that \(f\) is \(L_{f,h}\)-smooth relative to \(h\) to make the smoothness modulus \(L_{f,h}\) explicit.
\end{definition}


Note that the constant \(L_{f,h}\) may be loose.
For instance, if \(f\) is convex, then \(Lh+f\) is convex for any \(L\geq0\).
This motivates us to consider one-sided conditions and treat \(f\) and \(-f\) separately.

\begin{definition}[relative weak convexity]\label{def:relweak}%
	We say that an lsc function \(\func{f}{\R^n}{\Rinf}\) is weakly convex relative to \(h\) if there exists a (possibly negative) constant \(\sigma_{f,h}\in\R\) such that \(f-\sigma_{f,h}h\) is a convex function.
	We may alternatively say that \(f\) is \(\sigma_{f,h}\)-weakly convex relative to \(h\) to make the weak convexity modulus \(\sigma_{f,h}\) explicit.
\end{definition}
	
Note that relative smoothness of \(f\) is equivalent to the relative weak convexity of \(\pm f\).
Indeed, if \(f\) is \(L_{f,h}\)-smooth relative to \(h\), then both \(f\) and \(-f\) are \((-L_{f,h})\)-weakly convex.
Conversely, if \(f\) and \(-f\) are \(\sigma_{f,h}\)- and \(\sigma_{-f,h}\)-weakly convex relative to \(h\), respectively, then \(f\) (as well as \(-f\)) is \(L_{f,h}\)-smooth relative to \(h\) with
\begin{equation}\label{eq:Lf}
	L_{f,h}=\max\set{|\sigma_{f,h}|,|\sigma_{-f,h}|}.
\end{equation}
The relative weak convexity moduli \(\sigma_{\pm f,h}\) will be henceforth adopted when referring to \cref{ass:f}.
It will be convenient to \emph{normalize} these quantities into pure numbers
\begin{equation}\label{eq:pf}
	p_{\pm f,h}
{}\coloneqq{}
	\tfrac{\sigma_{\pm f,h}}{L_{f,h}}
{}\in{}
	[-1,1].
\end{equation}
To make all definitions and implications well posed, we will ignore the uninsteresting case in which \(f\) is affine, in which case \(L_{f,h}=0\) for any \(h\).
The comment below will be instrumental in \cref{sec:bounds}.

\begin{remark}\label{thm:p}%
	Invoking~\eqref{eq:Lf} and~\eqref{eq:pf} yields that
	\begin{equation}\label{eq:p+p}
		-2\leq p_{f,h}+p_{-f,h}\leq 0
	\quad\text{and}\quad
		-1\in\set{p_{f,h},p_{-f,h}},
	\end{equation}
	where the second inequality owes to the fact that, by definition, both \(f-\sigma_{f,h}h\) and \(-f-\sigma_{-f,h}h\) are convex functions, and therefore so is their sum \(-(\sigma_{f,h}+\sigma_{-f,h})h=-L_{f,h}(p_{f,h}+p_{-f,h})h\).
	Thus, whenever \(f\) is \emph{convex (resp. concave)}, \emph{one can take} \(p_{f,h}=0\) (resp. \(p_{-f,h}=0\)) and by virtue of the inclusion in \eqref{eq:p+p} it directly follows that \(p_{-f,h}=-1\) (resp. \(p_{f,h}=-1\)).
\end{remark}



We now turn to a lemma that guarantees well definedness of \cref{alg:BiFRB}.

\begin{lemma}[relative prox-boundedness]\label{thm:proxbounded}%
	Suppose that \cref{ass:basic} holds.
	Then, the set \(\argmin\set{\gamma g+h+\innprod{v}{\cdot}}\) as in \cref{ass:T} is nonempty for any \(v\in\R^n\) and \(0<\gamma<\nicefrac{1}{[\sigma_{-f,h}]_-}\).
	In other words, \(g\) is prox bounded relative to \(h\) with threshold \(\gamma_{g,h}\geq\nicefrac{1}{[\sigma_{-f,h}]_-}\) \cite[Def. 2.3]{kan2012moreau}.
\end{lemma}
\begin{proof}
	Recall that a proper convex function admits affine minorant; see, e.g,~\cite[Cor. 16.18]{bauschke2017convex}.
	It then follows that \(\gamma g+h\) is 1-coercive by observing that
	\[
		\gamma g+h
	{}={}
		\gamma\cost-\gamma f+h
	{}={}
		\gamma\underbracket*{
			\vphantom{\left([\sigma_{-f,h}]_-\right)}
			\cost
		}_{\mathclap{\geq\inf\cost}}
		{}+{}
		\gamma\underbracket*{
			\left(-f+[\sigma_{-f,h}]_-h\right)
		}_{\text{convex}}
		{}+{}
		\underbracket*{
			\left(1-\gamma[\sigma_{-f,h}]_-\right)
		}_{>0}
		\overbracket*{
			\,h\,
		}^{\text{\clap{1-coercive}}}
	\]
	on \(\closure C\) and \(\infty\) on \(\R^n\setminus\closure C\).
\end{proof}

We point out that the content of this subsection is a (well-known) generalization of the (well-known) equivalence between smoothness relative to the Euclidean kernel \(\j\coloneqq(\nicefrac12)\|{}\cdot{}\|^2\) and Lipschitz differentiability, a fact that will be invoked in \cref{sec:toolbox} and whose proof is given next for the sake of completeness.

\begin{lemma}[Lipschitz smoothness from weak convexity]\label{thm:Lipsmooth}%
	Let \(\j\coloneqq(\nicefrac12)\|{}\cdot{}\|^2\).
	Then, for any \(\func{F}{\R^n}{\Rinf}\) the following are equivalent:
	\begin{enumerateq}
	\item\label{thm:j:weakcvx}%
		there exist \(\sigma_{\pm F}\in\R\) such that both \(F-\sigma_F\j\) and \(-F-\sigma_{-F}\j\) are proper, convex and lsc;
	\item\label{thm:j:innprod}%
		\(\dom\partial F=\R^n\), and there exist \(\sigma_{\pm F}\in\R\) such that for all \((x_i,v_i)\in\graph\partial F\), \(i=1,2\), it holds that
		\(
			\sigma_F\|x_1-x_2\|^2
		{}\leq{}
			\innprod{v_1-v_2}{x_1-x_2}
		{}\leq{}
			-\sigma_{-F}\|x_1-x_2\|^2
		\);
	\item\label{thm:j:C11}%
		\(\nabla F\) is \(L_F\)-Lipschitz differentiable for some \(L_F\geq0\).
	\end{enumerateq}
	In particular, \ref{thm:j:weakcvx} and/or \ref{thm:j:innprod} imply \ref{thm:j:C11} with \(L_F=\max\set{|\sigma_F|,|\sigma_{-F}|}\), and conversely \ref{thm:j:C11} implies \ref{thm:j:weakcvx} and \ref{thm:j:innprod} with \(\sigma_{\pm F}=-L_F\).%
	\begin{proof}~
		\begin{proofitemize}
		\item \ref{thm:j:weakcvx} \(\Leftrightarrow\) \ref{thm:j:innprod}~
			Follows from \cite[Ex.s 12.28(b),(c)]{rockafellar2011variational}.
		\item \ref{thm:j:innprod} \(\Rightarrow\) \ref{thm:j:C11}~
			Invoking again \cite[Ex. 12.28(c)]{rockafellar2011variational} yields that both \(F\) and \(-F\) are lower-\(\C^2\), in the sense of \cite[Def. 10.29]{rockafellar2011variational}, and continuous differentiability then follows from \cite[Prop. 10.30]{rockafellar2011variational}.
			Thus, \(v_i=\nabla F(x_i)\) in assertion \ref{thm:j:innprod}, and denoting \(L_F=\max\set{|\sigma_F|,|\sigma_{-F}|}\) the function \(\psi\coloneqq L_F\j+F\) satisfies
			\[
				0
			{}\leq{}
				\innprod{\nabla(\psi(x_1)-\nabla\psi(x_2)}{x_1-x_2}
			{}\leq{}
				2L_F\|x_1-x_2\|^2.
			\]
			It then follows from \cite[Thm. 2.1.5]{nesterov2018lectures} that \(\psi\) is convex and \((2L_F)\)-Lipschitz differentiable, and that
			\[
				\tfrac{1}{2L_F}
				\|\nabla\psi(x_1)-\nabla\psi(x_2)\|^2
			{}\leq{}
				\innprod{\nabla\psi(x_1)-\nabla\psi(x_2)}{x_1-x_2}.
			\]
			Using the definition of \(\psi\) and expanding the squares yields
			\(
				\tfrac{1}{2L_F}\|\nabla F(x_1)-\nabla F(x_2)\|^2
			{}\leq{}
				\tfrac{L_F}{2}\|x_1-x_2\|^2
			\),
			proving that \(\nabla F\) is \(L_F\)-Lipschitz continuous.
		\item \ref{thm:j:C11} \(\Rightarrow\) \ref{thm:j:weakcvx}~
			From the quadratic upper bound \cite[Prop. A.24]{bertsekas2016nonlinear} it follows that
			\[
				\pm F(x_2)
			{}\geq{}
				\pm F(x_1)
				{}\pm{}
				\innprod{\nabla F(x_1)}{x_2-x_1}
				{}-{}
				\tfrac{L_F}{2}\|x_2-x_1\|^2.
			\]
			or, equivalently,
			\[
				(L_F\j\pm F)(x_2)
			{}\geq{}
				(L_F\j\pm F)(x_1)
				{}+{}
				\innprod{\nabla(L_F\j\pm F)(x_1)}{x_2-x_1}.
			\]
			This proves convexity of \(L_F\j\pm F\), whence the claim by taking \(\sigma_{\pm F}=-L_F\).
		\qedhere
		\end{proofitemize}
	\end{proof}
\end{lemma}

	\section{Algorithmic analysis toolbox}\label{sec:toolbox}
		In the literature, convergence analysis for nonconvex splitting algorithms typically revolves around the identification of a `Lyapunov potential', namely, a lower bounded function that decreases its value along the iterates.
In this section we will pursue this direction.
We will actually go one step further in identifying a function which, additionally to serving as Lyapunov potential, is also \emph{continuous}.
This property, whose utility will be showcased in \cref{sec:CLyD}, will come as the result of a parametric minimization, as discussed in the following two subsections.

In what follows, to simplify the discussion we introduce
\begin{equation}
	\h\coloneqq\tfrac1\gamma h-f
\quad\text{and}\quad
	\f\coloneqq f-\tfrac\beta\gamma h,
\end{equation}
and observe that \(\h\) too is a Legendre kernel, for \(\gamma\) small enough.

\begin{lemma}[{{\cite[Thm. 4.1]{ahookhosh2021bregman}}}]\label{thm:Legendre}%
	Suppose that \cref{ass:f} holds.
	Then, for every \(\gamma<\nicefrac{1}{[\sigma_{-f,h}]_-}\) the function \(\h\) is a Legendre kernel with \(\dom\h=\dom h\).
\end{lemma}

	We will also (ab)use the notation \(\D_\psi\) of the Bregman distance for functions \(\func{\psi}{\R^n}{\Rinf}\) differentiable on \(C\) that are not necessarily convex, thereby possibly having \(\D_\psi\not\geq0\).
	This notational abuse is justified by the fact that all algebraic identities of the Bregman distance used in the manuscript (e.g., the three-point identity \cite[Lem. 3.1]{chen1993convergence}) are valid regardless of whether \(\psi\) is convex or not, and will overall yield a major simplification of the math.

		\subsection{Parametric minimization model}\label{sec:model}
		As a first step towards the desired goals, as well as to considerably simplify the discussion, we begin by observing that the \BiFRB-update is the result of a parametric minimization.
Namely, by introducing the ``model''
\begin{subequations}
	\begin{align}
	\label{eq:M-innprod}
		\M(w;x,x^-)
	{}={} &
		\varphi(w)
		{}+{}
		\D_{\h}(w,x)
		{}+{}
		\innprod*{w-x}{\nabla\f(x)-\nabla\f(x^-)}
	\\
	\label{eq:M-D}
	{}={} &
		\varphi(w)
		{}+{}
		\D_{\h-\f}(w,x)
		{}+{}
		\D_{\f}(w,x^-)
		{}-{}
		\D_{\f}(x,x^-),
	\end{align}
\end{subequations}
observe that the \(x\)-update in \BiFRB\ can be compactly expressed as
\begin{subequations}\label{subeq:x+}
	\begin{align}
		x^{k+1}
	{}\in{} &
		\T(x^k,x^{k-1}),
	\shortintertext{%
		where \(\ffunc{\T}{C\times C}{C}\) defined by
	}
	\label{eq:T}
		\T(x,x^-)
	{}\coloneqq{} &
		\argmin_{w\in\R^n}\M(w;x,x^-)
	\end{align}
\end{subequations}
is the \BiFRB-operator with stepsize \(\gamma\) and inertial parameter \(\beta\).
The fact that \(\T\) maps pairs in \(C\times C\) to subsets of \(C\) is a consequence of \cref{ass:T}, as we are about to formalize in \cref{thm:T:osc}.
Note that many \(\M\) can be defined giving rise to the same \(\T\), and all these differ by additive terms which are constant with respect to \(w\).
Among these, the one given in \eqref{eq:M-D} reflects the tangency condition \(\M(x;x,x^-)=\varphi(x)=\cost(x)\) for every \(x,x^-\in C\).
A consequence of this fact and other basic properties are summarized next.

\begin{lemma}[basic properties of the model \(\M\) and the operator \(\T\)]\label{thm:basic}%
	Suppose that \cref{ass:basic} holds, and let \(\gamma<\nicefrac{1}{[\sigma_{-f,h}]_-}\) and \(\beta\in\R\) be fixed.
	The following hold:
	\begin{enumerate}
	\item \label{thm:M:eq}%
		\(
			\M(x;x,x^-)
		{}={}
			\cost(x)
		\)
		for all \(x,x^-\in C\).
	\item \label{thm:M:lb}%
		\(\M(w;x,x^-)\) is level bounded in \(w\) locally uniformly in \((x,x^-)\) on \(C\times C\).\footnote{%
			Namely, \(\bigcup_{(x,x^-)\in\mathcal V}\set{w\in\R^n}[\M(w;x,x^-)\leq\alpha]\) is bounded for any \(\mathcal V\subset C\times C\) compact and \(\alpha\in\R\).%
		}
	\item \label{thm:T:osc}%
		\(\T\) is locally bounded and osc,\footnote{%
			Being \(\T\) defined on \(C\times C\), osc and local boundedness are meant relative to \(C\times C\).
			Namely, \(\graph\T\) is closed relative to \(C\times C\times\R^n\), and \(\bigcup_{(x,x^-)\in\mathcal V}\T(x,x^-)\) is bounded for any \(\mathcal V\subset C\times C\) compact.%
		}
		and \(T(x,x^-)\) is a nonempty and compact subset of \(C\) for any \(x,x^-\in C\).
	\item \label{thm:T:partial}%
		\(\nabla\h(x)-\nabla\h(\bar x)-\nabla\f(x)+\nabla\f(x^-)\in\hat\partial\varphi(\bar x)\) for any \(x,x^-\in C\) and \(\bar x\in\T(x,x^-)\).
	\item \label{thm:T:fix}%
		If \(x\in\T(x,x)\), then \(0\in\hat\partial\varphi(x)\) and \(\def\stepsize{\gamma'}\T(x,x)=\set{x}\) for every \(\gamma'\in(0,\gamma)\).
	\end{enumerate}
	\begin{proof}
		We start by observing that \cref{thm:proxbounded} ensures that the set \(\T(x,x^-)\) is non\-empty for any \(x,x^-\in C\); this follows by considering the expression \eqref{eq:M-innprod} of the model, by observing that, for any \(x\in C\),
		\(
			\varphi+\D_{\h}({}\cdot{},x)
		{}={}
			g+\tfrac1\gamma h-\h(x)-\innprod*{\nabla\h(x)}{{}\cdot{}-x}
		\).
		For the same reason, it then follows from \cref{ass:T} that \(T(x,x^-)\subset C\).
		\begin{proofitemize}
		\item \ref{thm:M:eq}~
			Apparent, by considering \(w=x\) in \eqref{eq:M-D}.
		\item \ref{thm:M:lb} \& \ref{thm:T:osc}~
			The first assertion owes to the fact that \(\h\) is 1-coercive by \cref{thm:Legendre} and that both \(\h\) and \(\nabla\f\) are continuous on \(C\), so that for any compact \(\mathcal V\subset\R^n\times\R^n\) one has that
			\(
				\lim_{\|w\|\to\infty}\inf_{(x,x^-)\in\mathcal V}\M(w;x,x^-)=\infty
			\),
			as is apparent from \eqref{eq:M-innprod}.
			In turn, the second assertion follows from \cite[Thm. 1.17]{rockafellar2011variational}.
		\item \ref{thm:T:partial}~
			Follows from the optimality conditions of \(\bar x\in\argmin\M({}\cdot{};x,x^-)\), for \(\bar x\in C\) by assumption and thus the calculus rule of \cite[Ex. 8.8(c)]{rockafellar2011variational} applies (having \(\h\) smooth around \(\bar x\in C\)).
		\item \ref{thm:T:fix}~
			That \(0\in\hat\partial\varphi(x)\) follows from assertion \ref{thm:T:partial}, and the other claim from \cite[Lem. 3.6]{ahookhosh2021bregman} by observing that \(\T(x,x)=\argmin\set{\varphi+\D_{\h}({}\cdot{},x)}\) for any \(\gamma>0\) and \(\beta\in\R\).
		\qedhere
		\end{proofitemize}
	\end{proof}
\end{lemma}

\begin{remark}[inertial effect]%
	Letting \(\tilde f=f+ch\) and \(\tilde g=g-ch\) for some \(c\in\R\), \(\tilde f+\tilde g\) gives an alternative decomposition of \(\varphi\) which still complies with \cref{ass:basic}, having \(\sigma_{\pm\tilde f,h}=\sigma_{\pm f,h}\pm c\).
	Relative to this decomposition, it is easy to verify that the corresponding model \(\tilde{\mathcal M}^{h\text{-\sc frb}}\) is related to the original one in \eqref{eq:M-innprod} as
	\[
		{\tildeall\M}
	{}={}
		\M
	\quad\text{with}\quad
		\begin{cases}
			\gamma ={} & \frac{\tilde\gamma}{1-\tilde\gamma c}\\[4pt]
			\beta ={} & \frac{\tilde\beta-\tilde\gamma c}{1-\tilde\gamma c}
		\end{cases}
	~~\Leftrightarrow~~~
		\begin{cases}
			\tilde\gamma ={} & \frac{\gamma}{1+\gamma c}\\[4pt]
			\tilde\beta ={} & \frac{\beta+\gamma c}{1+\gamma c},
		\end{cases}
	\]
	and in particular \BiFRB\ steps with the respective parameters coincide.
	The effect of inertia can then be explained as a redistribution of multiples of \(h\) among \(f\) and \(g\) in the problem formulation, having
	\(
		\M
	{}={}
		{%
			\def\model{\tilde{\mathcal M}}
			\def\stepsize{\frac{\gamma}{1-\beta}\,}
			\def\inertia{\!0}
			\M
		}
	\)
	for any \(\gamma>0\) and \(\beta<1\).
\end{remark}

		\subsection{The \texorpdfstring{$i^*$}{i*}FRB-envelope}%
			Having defined model $\M$ and its solution mapping $\T$ resulted from parametric minimization, we now introduce the associated value function, which we name \BiFRB-envelope.

\begin{definition}[\BiFRB-envelope]\label{def: envelope}%
	The envelope function associated to \BiFRB\ with stepsize \(\gamma<\nicefrac{1}{[\sigma_{f,h}]_-}\)  and inertia \(\beta\in\R\) is \(\func{\env}{C\times C}{\R}\) defined as
	\begin{equation}\label{eq:env}
		\env(x,x^-)
	{}\coloneqq{}
		\inf_{w\in\R^n}\M(w;x,x^-).
	\end{equation}
\end{definition}


\begin{lemma}[basic properties of \(\env\)]%
	Suppose that \cref{ass:basic} holds.
	Then, for any \(\gamma<\nicefrac{1}{[\sigma_{f,h}]_-}\) and \(\beta\in\R\) the following hold:
	\begin{enumerate}
	\item\label{thm:env:C0}%
		\(\env\) is (real-valued and) continuous on \(C\times C\); in fact, it is locally Lipschitz provided that \(f,h\in\C^2(C)\).
	\item\label{thm:env:eq}%
		For any \(x,x^-\in C\) and \(\bar x\in\T(x,x^-)\)
		\[
		\ifsiam
			\mathtight[0.45]
		\fi
			\env(x,x^-)
		{}={}
			\M(\bar x;x,x^-)
		{}={}
			\varphi(\bar x)
			{}+{}
			\D_{\h-\f}(\bar x,x)
			{}+{}
			\D_{\f}(\bar x,x^-)
			{}-{}
			\D_{\f}(x,x^-).
		\]
	\item\label{thm:env:leq}%
		\(
			\env(x,x^-)
		{}\leq{}
			\varphi(x)
		\)
		for any \(x,x^-\in C\).
	\end{enumerate}
	\begin{proof}~
		\begin{proofitemize}
		\item \ref{thm:env:C0}~
			In light of the uniform level boundedness asserted in \cref{thm:M:lb}, continuity of \(\env\) follows from $\text{\cite[Thm. 1.17(c)]{rockafellar2011variational}}$ by observing that the mapping \((x,x^-)\mapsto\M(w;x,x^-)\) is continuous for every \(w\);
			in fact, when \(f\) and \(h\) are both \(\C^2\) on \(C\), its gradient exists and is continuous with respect to all its arguments, which together with local boundedness of \(\T\), cf. \cref{thm:T:osc}, gives that \(-\env\) is a lower-\(\C^1\) function in the sense of \cite[Def. 10.29]{rockafellar2011variational}, and in particular locally Lipschitz continuous by virtue of \cite[Thm.s 10.31 and 9.2]{rockafellar2011variational}.
		\item \ref{thm:env:eq} \& \ref{thm:env:leq}~
			The identity follows by definition, cf. \eqref{eq:env} and \eqref{eq:T}.
			The inequality follows by considering \(w=x\) in \eqref{eq:env} and \eqref{eq:M-D}.
		\qedhere
		\end{proofitemize}
	\end{proof}
\end{lemma}

		\subsection{Establishing a merit function}%
			We now work towards establishing a merit function for \BiFRB, starting from comparing the values of \(\env(\bar x,x)\) and \(\env(x,x^-)\), with \(\bar x\in\T(x,x^-)\).
Owing to \cref{thm:env:leq}, we have
\begin{align}
\nonumber
	\env(\bar x,x)
{}\leq{}
	\varphi(\bar x)
{}={} &
	\cost(\bar x)
\\
\label{eq:descent}
{}={} &
	\env(x,x^-)
	{}-{}
	\D_{\h-\f}(\bar x,x)
	{}-{}
	\overbracket*{
		\D_{\f}(\bar x,x^-)
	}
	{}+{}
	\D_{\f}(x,x^-).
\end{align}
	From here two separate cases can be considered, each yielding surprisingly different results.
	The watershed lies in whether the bracketed term is positive or not: one case will result in a very straightforward convergence analysis in the full generality of \cref{ass:basic}, while the other will necessitate an additional Lipschitz differentiability requirement.
	The convergence analysis in both cases revolves around the identification of a constant \(c>0\) determining a lower bounded merit function
	\begin{equation}\label{eq:LL}
		\LL=\env+\tfrac{c}{2\gamma}\D+\D_\xi.
	\end{equation}
	The difference between the two cases is determined by function \(\xi\) appearing in the last Bregman operator \(\D_\xi\), having \(\xi=\f\) in the former case and \(\xi=L_{\f}\j\) in the latter, where \(L_{\f}\) is a Lipschitz constant for \(\nabla\f\) and
	\begin{equation}\label{eq:j}
		\j\coloneqq\tfrac12\|{}\cdot{}\|^2
	\end{equation}
	is the squared Euclidean norm.
The two cases are stated in the next theorem, which constitutes the main result of this section.
Special and worst-case scenarios leading to simplified statements will be given in \cref{sec:bounds}.
In what follows, patterning the normalization of \(\sigma_{\pm f,h}\) into \(p_{\pm f,h}\) detailed in \cref{sec:relsmooth} we also introduce the scaled stepsize
\begin{equation}
	\alpha\coloneqq\gamma L_{f,h},
\end{equation}
which as a result of the convergence analysis will be confined in the interval \((0,1)\).

\begin{theorem}\label{thm:SD}%
	Suppose that \cref{ass:basic} holds and consider one of the following scenarios:
	\begin{enumerator}
	\item\label{thm:SD:fconv}%
		either \(\f\) is convex (\ie, \(\alpha p_{f,h}-\beta\geq0\)) and
		\(
			\beta
		{}>{}
			-\nicefrac{(1+3\alpha p_{-f,h})}{2}
		\),
		in which case
		\[
			\xi\coloneqq\f
		\quad\text{and}\quad
				c
			{}\coloneqq{}
				1+2\beta+3\alpha p_{-f,h}>0,
		\]
	\item\label{thm:SD:fsmooth}%
		or \(\f\) is \(L_{\f}\)-Lipschitz differentiable, \(h\) is \(\sigma_h\)-strongly convex, and
		\[
			c
		{}\coloneqq{}
			(1+\alpha p_{-f,h})\sigma_h
			{}-{}
			2\gamma L_{\f}
		{}>{}
			0,
		\]
		in which case \(\xi\coloneqq L_{\f}\j\).
	\end{enumerator}
	Then, for \(\LL\) as in \eqref{eq:LL} the following assertions hold:
	\begin{enumerate}
	\item\label{thm:SD:leq}%
		\begin{subequations}
			For every \(x,x^-\in C\) and \(\bar x\in\T(x,x^-)\),
			\begin{align}
			\label{eq:SD}
				\LL(\bar x,x)
			{}\leq{} &
				\LL(x,x^-)
				{}-{}
				\tfrac{c}{2\gamma}\D(\bar x,x)
				{}-{}
				\tfrac{c}{2\gamma}\D(x,x^-)
			\shortintertext{and}
			\label{eq:L:geq}
				\cost(\bar x)
			{}\leq{} &
				\LL(x,x^-)
				{}-{}
				\tfrac{c}{\gamma}\D(\bar x,x)
				{}-{}
				\tfrac{c}{2\gamma}\D(x,x^-).
			\end{align}
		\end{subequations}
	\item\label{thm:SD:inf}%
		\(\inf\LL=\inf\cost\), and \(\LL\) is level bounded iff so is \(\cost\).
	\end{enumerate}
\end{theorem}

The proof of this result is detailed in the dedicated \cref{sec:SD:proof}; before that, let us draw some comments.
As clarified in the statement of \cref{thm:SD:fconv}, convexity of \(\f\) can be enforced by suitably choosing \(\gamma\) and \(\beta\) without imposing additional requirements on the problem. However, an unusual yet reasonable condition on inertial parameter \(\beta\) may be necessary.

\begin{remark}\label{rem:negative inertia}%
	In order to furnish \cref{thm:SD:fconv}, we shall see soon that $\beta\leq0$ may be required; see \cref{sec:bounds}.
	Such assumption, although more pessimistic, coincides with a recent conjecture by Dragomir et al.~\cite[\S4.5.3]{dragomir2022optimal}, which states that inertial methods with nonadaptive coefficients fail to convergence in the relative smoothness setting, and provides an alternative perspective to the same matter through the lens of the convexity of $\f$.
\end{remark}

Unlike \cref{thm:SD:fconv}, however, additional assumptions are needed for the Lipschitz differentiable case of \cref{thm:SD:fsmooth}. 
This is because the requirement is equivalent to smoothness relative to the Euclidean Bregman kernel \(\j\), while \cref{ass:basic} prescribes bounds only relative to \(h\).
For simplicity, for functions \(\func{\psi_1,\psi_2}{\R^n}{\Rinf}\) we write
\[
	\psi_1 \preceq \psi_2
\]
to indicate that \(\psi_2-\psi_1\) is convex.

\begin{remark}\label{thm:fsmooth}%
	Under \cref{ass:basic}, one has that \(\f\) is Lipschitz-differentiable with modulus \(L_{\f}\) under either one of the following conditions:
	\begin{enumerator}[B*]
	\item\label{thm:fsmooth:hsmooth}%
		either \(\nabla h\) is \(L_h\)-Lipschitz, in which case
		\(
			L_{\f}
		{}={}
			\tfrac{L_h}{\gamma}\max\set{
				\beta-\alpha p_{f,h},
				-\beta-\alpha p_{-f,h}
			}
		\),%
	\item\label{thm:fsmooth:beta0}%
		or \(\beta=0\) and \(\nabla f\) is \(L_f\)-Lipschitz, in which case \(L_{\f}=L_f\).
	\end{enumerator}
	Recalling that \(\f=f-\frac\beta\gamma h\), the second condition is tautological.
	In case \(\nabla h\) is \(L_h\)-Lipschitz, \cref{thm:Lipsmooth} yields that
	\[
		-\tfrac{L_h}{\gamma}[\beta-\alpha p_{f,h}]_+\j
	{}\preceq{}
		\tfrac{\alpha p_{f,h}-\beta}{\gamma}h
	{}\preceq{}
		\f
	{}\preceq{}
		-\tfrac{\alpha p_{-f,h}+\beta}{\gamma}h
	{}\preceq{}
		\tfrac{L_h}{\gamma}[-\beta-\alpha p_{-f,h}]_+\j,
	\]
	and that consequently
	\[
		L_{\f}
	{}\coloneqq{}
		\tfrac{L_h}{\gamma}\max\set{
			\bigl[\beta-\alpha p_{f,h}\bigl]_+,
			\bigl[-\beta-\alpha p_{-f,h}\bigr]_+
		}
	{}={}
		\tfrac{L_h}{\gamma}\max\set{
			\beta-\alpha p_{f,h},
			-\beta-\alpha p_{-f,h}
		}
	\]
	is a Lipschitz modulus for \(\nabla\f\).
\end{remark}

		\subsection{Simplified bounds}\label{sec:bounds}%
			In this section we provide bounds that only discern whether \(f\) is convex, concave, or neither of the above.
As discussed in \cref{thm:p}, these cases can be recovered by suitable combinations of the coefficients \(p_{\pm f,h}\in\set{0,\pm1}\), and thus lead to easier, though possibly looser, bounds compared to those in \cref{thm:SD}.
We will also avail ourselves of the estimates of \(L_{\f}\) in \cref{thm:fsmooth} to discuss the cases in which \(\f\) is Lipschitz differentiable.

Without distinguishing between upper and lower relative bounds, whenever \(f\) is \(L_{f,h}\)-smooth relative to \(h\) as in \cref{ass:basic} one can consider \(\sigma_{\pm f,h}=-L_{f,h}\) or, equivalently, \(p_{f,h}=p_{-f,h}=-1\).
Plugging these values into \cref{thm:SD} yields the following.

\begin{corollary}[worst-case bounds]\label{thm:bounds:wc}%
	Suppose that \cref{ass:basic} holds.
	All the claims of \cref{thm:SD} hold when \(\gamma>0\), \(\beta\in\R\) and \(c>0\) are such that
	\begin{enumerator}
	\item\label{thm:bounds:wc:fconv}%
		either
	\(
			-\nicefrac{1}{2}<\beta<0
		\)
and
		\(	\gamma\leq(\nicefrac{1}{L_{f,h}})\min\set{-\beta,\nicefrac{(1+2\beta-c)}{3}},
		\)
		in which case \(\xi=\f\);
	\end{enumerator}
	\begin{enumerator}[B]
	\item\label{thm:bounds:wc:hsmooth}%
		or \(h\) is \(\sigma_h\)-strongly convex and \(L_h\)-Lipschitz differentiable,
	\(
			|\beta|<\nicefrac{\sigma_h}{2L_h}
		\)
	and
       \(
			\gamma\leq(\nicefrac{1}{L_{f,h}})[\nicefrac{(\sigma_h-2L_h|\beta|-c)}{(\sigma_h+2L_h)}],
		\)
		in which case \(\xi=(\nicefrac{L_h}{\gamma})(\alpha+|\beta|)\j\);
	\item\label{thm:bounds:wc:beta0}%
		or \(h\) is \(\sigma_h\)-strongly convex, \(\nabla f\) is \(L_f\)-Lipschitz,
	\(
			\beta=0
		\)
and 
    \(
			\gamma\leq(\nicefrac{1}{L_{f,h}})[\nicefrac{(\sigma_h-c)}{(\sigma_h+2L_h)}],
   \)
		in which case \(\xi=L_f\j\).
	\end{enumerator}
	\begin{proof}
		We will discuss only the bounds on \(\gamma\) as those on \(\beta\) will automatically follow from the fact that \(\gamma>0\).
		Setting \(p_{\pm f,h}=-1\) in \cref{thm:SD}, one has:%
		\begin{proofitemize}
		\item\ref{thm:bounds:wc:fconv}~
			From \cref{thm:SD:fconv} we otain
			\(
				0
			{}<{}
				c
			{}={}
				1+2\beta-3\alpha
			\)
			and
			\(
				-\nicefrac{(1-3\alpha)}{2}
			{}<{}
				\beta
			{}\leq{}
				-\alpha
			\).
			By replacing the first equality with ``\(\leq\)'', that is, by possibly allowing looser values of \(c\), the bound on \(\gamma=\nicefrac{\alpha}{L_{f,h}}\) as in \ref{thm:bounds:wc:fconv} is obtained.
		\item\ref{thm:bounds:wc:hsmooth} \& \ref{thm:bounds:wc:beta0}~
			the two subcases refer to the corresponding items in \cref{thm:fsmooth}.
			We shall show only the first one, as the second one is a trivial adaptation.
			The value of \(L_{\f}\) as in \cref{thm:fsmooth:hsmooth} reduces to \(L_{\f}=(\nicefrac{L_h}{\gamma})(\alpha+|\beta|)\).
			Plugged into \cref{thm:SD:fsmooth} yields
			\(
				0
			{}<{}
				c
			{}={}
				(1-\alpha)\sigma_h
				{}-{}
				2L_h(\alpha+|\beta|)
			\),
			implying that \(\gamma\) is bounded as in \ref{thm:bounds:wc:hsmooth}.
		\qedhere
		\end{proofitemize}
	\end{proof}
\end{corollary}

When \(f\) is convex, \(\sigma_{f,h}=0\) can be considered resulting in \(p_{f,h}=0\) and \(p_{-f,h}=-1\).

\begin{corollary}[bounds when \(f\) is convex]\label{thm:bounds:convex}%
	Suppose that \cref{ass:basic} holds and that \(f\) is convex.
	All the claims of \cref{thm:SD} remain valid if \(\gamma>0\), \(\beta\in\R\) and \(c>0\) are such that
	\begin{enumerator}
	\item\label{thm:bounds:convex:fconv}%
	either	
		\(
			-\nicefrac12<\beta\leq0
		\)
		and
		\(	\gamma\leq(\nicefrac{1}{L_{f,h}})[\nicefrac{(1+2\beta-c)}{3}],
		\)
		in which case \(\xi=\f\);
	\end{enumerator}
	\begin{enumerator}[B]
	\item\label{thm:bounds:convex:hsmooth}%
		or \(h\) is \(\sigma_h\)-strongly convex and \(L_h\)-Lipschitz differentiable,
		\[
			|\beta|<\tfrac{\sigma_h}{2L_h}
		\quad\text{and}\quad
			\gamma\leq\tfrac{1}{L_{f,h}}\min\set{
				\tfrac{\sigma_h+2L_h\beta-c}{\sigma_h+2L_h},\,
				\tfrac{\sigma_h-2L_h\beta-c}{\sigma_h}
			},
		\]
		in which case \(\xi=(\nicefrac{L_h}{\gamma})\max\set{\beta,\alpha-\beta}\j\);
	\item\label{thm:bounds:convex:beta0}%
		or \(h\) is \(\sigma_h\)-strongly convex, \(\nabla f\) is \(L_f\)-Lipschitz,
		\(\beta=0\), and
		\(
			\gamma\leq\nicefrac{(\sigma_h-c)}{(\sigma_hL_{f,h}+2L_f)},
		\)
		in which case \(\xi=L_f\j\).
	\end{enumerator}
	\begin{proof}
		As motivated in the proof of \cref{thm:bounds:wc}, it suffices to prove the bounds on \(\gamma\).
		Similarly, the proof of assertion \ref{thm:bounds:convex:beta0} is an easy adaptation of that of \ref{thm:bounds:convex:hsmooth}.
		Setting \(p_{f,h}=0\) and \(p_{-f,h}=-1\) in \cref{thm:SD}, one has:
		\begin{proofitemize}
		\item\ref{thm:bounds:convex:fconv}~
			From \cref{thm:SD:fconv} we otain
			\(
				0
			{}<{}
				c
			{}={}
				1+2\beta-3\alpha
			\)
			and
			\(
				-\nicefrac{(1-3\alpha)}{2}
			{}<{}
				\beta
			{}\leq{}
				0
			\).
			Again by possibly allowing looser values of \(c\), the bound on \(\gamma=\nicefrac{\alpha}{L_{f,h}}\) as in \ref{thm:bounds:convex:fconv} is obtained.
		\item\ref{thm:bounds:convex:hsmooth}~
			The value of \(L_{\f}\) as in \cref{thm:fsmooth:hsmooth} reduces to \(\frac{L_h}{\gamma}\max\set{\beta,\alpha-\beta}\).
			Plugged into \cref{thm:SD:fsmooth} yields
			\(
				0
			{}<{}
				c
			{}={}
				(1-\alpha)\sigma_h
				{}-{}
				2L_h\max\set{\beta,\alpha-\beta}
			\),
			which can be loosened as
			\[
				\begin{cases}
					c
				{}\leq{}
					(1-\alpha)\sigma_h
					{}-{}
					2L_h\beta
				\\
					c
				{}\leq{}
					(1-\alpha)\sigma_h
					{}-{}
					2L_h(\alpha-\beta).
				\end{cases}
			\]
			In terms of \(\gamma=\nicefrac{\alpha}{L_{f,h}}\), this results in the bound of assertion \ref{thm:bounds:convex:hsmooth}.
		\qedhere
		\end{proofitemize}
	\end{proof}
\end{corollary}

Similarly, when \(f\) is concave (that is, \(-f\) is convex), then \(\sigma_{-f,h}=0\) can be considered, resulting in \(p_{f,h}=-1\) and \(p_{-f,h}=0\).
The proof is omitted, as it uses the same arguments as in the previous results.

\begin{corollary}[bounds when \(f\) is concave]\label{thm:bounds:concave}%
	Suppose that \cref{ass:basic} holds and that \(f\) is concave.
	All the claims of \cref{thm:SD} remain valid if \(\gamma>0\), \(\beta\in\R\) and \(c>0\) are such that
	\begin{enumerator}
	\item\label{thm:bounds:concave:fconv}%
		either
\(
		\nicefrac{(c-1)}{2}\leq\beta<0
		\)
		and
		\(	
		\gamma\leq-\nicefrac{\beta}{L_{f,h}},
		\)
		in which case \(\xi=\f\);
	\item\label{thm:bounds:concave:hsmooth}%
		or \(h\) is \(\sigma_h\)-strongly convex and \(L_h\)-Lipschitz differentiable,
		\[
			\tfrac{c-\sigma_h}{2L_h}\leq\beta<\tfrac{\sigma_h}{2L_h}
		\quad\text{and}\quad
			\gamma\leq\tfrac{1}{L_{f,h}}\tfrac{\sigma_h-2L_h\beta-c}{2L_h},
		\]
		in which case \(\xi=\frac{L_h}{\gamma}\max\set{\alpha+\beta,-\beta}\j\);
	\item\label{thm:bounds:concave:beta0}%
		or \(h\) is \(\sigma_h\)-strongly convex, \(f\) is \(L_f\)-Lipschitz differentiable, \(\beta=0\) and 
     \(
			\gamma\leq\nicefrac{(\sigma_h-c)}{(2L_f)},
    \)
		in which case \(\xi=L_f\j\).
	\end{enumerator}
\end{corollary}

	\section{Convergence analysis}\label{sec:convergence}%
		In this section we study the behavior of sequences generated by \BiFRB.
Although some basic convergence results can be derived in the full generality of \cref{ass:basic}, establishing local optimality guarantees of the limit point(s) will ultimately require an additional full domain assumption.

\begin{assumption}\label{ass:h}
	Function \(h\) has full domain, that is, \(C=\R^n\).
\end{assumption}

\Cref{ass:h} is standard for nonconvex splitting algorithms in a relative smooth setting.
To the best of our knowledge, the question regarding whether this requirement can be removed remains open; see, e.g., \cite{teboulle2018simplified} and the references therein.

		\subsection{Function value convergence}
			We begin with the convergence of merit function value.

\begin{theorem}[function value convergence of \BiFRB]\label{thm:subseq}%
	Let $\seq{x^k}$ be a sequence generated by \BiFRB* in the setting of \cref{thm:SD}.
	Then,
	\begin{enumerate}
	\item\label{thm:subseq:SD}%
		It holds that
		\begin{equation}\label{SD inequality}
			\LL\left(x^{k+1},x^k\right)\leq\LL\left(x^k,x^{k-1}\right)-\tfrac{c}{2\gamma}\D\left(x^{k+1},x^k\right)-\tfrac{c}{2\gamma}\D\left(x^k,x^{k-1}\right).
		\end{equation}
		Then, $\sum_{k=0}^\infty \D\left(x^k,x^{k-1}\right)<\infty$ and $\LL(x^{k},x^{k-1})\to\varphi^\star$ as $k\to\infty$ for some $\varphi^\star\geq\inf\cost$.
	\item\label{thm:subseq:bounded}%
		If \(\cost\) is level bounded, then \(\seq{x^k}\) is bounded.
	\item\label{thm:subseq:omega}%
		Suppose that \cref{ass:h} holds, and let \(\omega\) be the set of limit points of \(\seq{x^k}\).
		Then, \(\varphi\) is constant on \(\omega\) with value \(\varphi^\star\), and for every \(x^\star\in\omega\) it holds that \(x^\star\in\T(x^\star,x^\star)\) and \(0\in\hat\partial\varphi(x^\star)\).
	\end{enumerate}
	\begin{proof}~
		\begin{proofitemize}
		\item\ref{thm:subseq:SD}~
			Recall from \cref{thm:SD} that \eqref{SD inequality} holds and that \(\inf\LL=\inf\cost>-\infty\), from which the convergence of \(\seq{\LL(x^{k},x^{k-1})}\) readily follows.
			In turn, a telescoping argument on \eqref{SD inequality} shows that \(\sum_{k\in\N}\D(x^k,x^{k-1})<\infty\).
		\item\ref{thm:subseq:bounded}~
			It follows from \cref{thm:subseq:SD} that \(\LL\left(x^{k+1},x^k\right)\leq\LL(x^0,x^{-1})\) holds for every \(k\).
			Then boundedness of \(\seq{x^k}\) is implied by level boundedness of \(\LL\); see \cref{thm:SD:inf}.
		\item\ref{thm:subseq:omega}~
			Suppose that a subsequence \(\seq{x^{k_j}}[j\in\N]\) converges to a point \(x^\star\), then so do \(\seq{x^{k_j\pm1}}\) by \cref{thm:subseq:SD} and \cite[Prop. 2.2(iii)]{bauschke2003iterating}.
			Since \(x^{k_j+1}\in\T(x^{k_j},x^{k_j-1})\), by passing to the limit osc of \(\T\) (\cref{thm:T:osc}) implies that \(x^\star\in\T(x^\star,x^\star)\).
			Invoking \cref{thm:T:fix} yields the stationarity condition \(0\in\hat\partial\varphi(x^\star)\).
			Moreover, by continuity of \(\LL\) one has
			\[
				\varphi^\star
			{}\defeq{}
				\lim_{k\to\infty}\LL\left(x^k,x^{k-1}\right)
			{}={}
				\LL(x^\star,x^\star)
			{}={}
				\varphi(x^\star),
			\]
			where the last equality follows from \cref{thm:env:eq}, owing to the inclusion \(x^\star\in\T(x^\star,x^\star)\) (and the fact that \(\D_\psi(x,x)=0\) for any differentiable function \(\psi\)).
			From the arbitrarity of \(x^\star\in\omega\) we conclude that \(\varphi\equiv\varphi^\star\) on \(\omega\).
		\qedhere
		\end{proofitemize}
	\end{proof}
\end{theorem}

It is now possible to demonstrate the necessity of some of the bounds on the stepsize that were discussed in \cref{sec:bounds}, by showing that \(\D(x^{k+1},x^k)\) may otherwise fail to vanish.
Note that, for \(\beta=0\), the following counterexample constitutes a tightness certificate for the bound \(\gamma<\nicefrac{1}{3L_f}\) derived in \cite{wang2022malitsky} in the noninertial Euclidean case.

\begin{example}\label{sharp stepsize example}%
	The bound \(\alpha=\gamma L_{f,h}<\nicefrac{(1+2\beta)}{3}\) is tight even in the Euclidean case.
	To see this, consider \(g=\indicator_{\set{\pm1}}\) and for a fixed \(L>0\) let \(f(x)=Lh(x)=\frac L2x^2\).
	Then, one has \(L_{f,h}=\sigma_{f,h}=L\) and \(\sigma_{-f,h}=-L\).
	For \(\gamma<\nicefrac{1}{L}=\nicefrac{1}{[\sigma_{-f,h}]_-}\), it is easy to see that
	\[
		\T(x,x^-)
	{}={}
		\sign\bigl(
			\nabla\h(x)-\nabla\f(x)+\nabla\f(x^-)
		\bigr)
	{}={}
		\sign\bigl(
			(1-2\alpha+\beta)x+(\alpha-\beta)x^-
		\bigr)
	\]
	(with \(\sign 0\coloneqq\set{\pm1}\)).
	If \(\alpha\geq\nicefrac{(1+2\beta)}{3}\), then \(\seq{(-1)^k}\) is a sequence generated by \BiFRB\ for which \(\D(x^{k+1},x^k)\equiv2\not\to0\).
\end{example}

As a consequence of \cref{thm:subseq:SD}, the condition \(\D(x^{k+1},x^k)\leq\varepsilon\) is satisfied in finitely many iterations for any tolerance \(\varepsilon>0\).
While this could be used as termination criterion, in the generality of \cref{ass:basic,ass:h} there is no guarantee on the relaxed stationarity measure \(\dist(0,\hat\partial\varphi(x^{k+1}))\), which through \cref{thm:T:partial} can only be estimated as
\begin{equation}\label{eq:subgrad}
	\dist(0,\hat\partial\varphi(x^{k+1}))
	{}\leq{}
	\|v^{k+1}\|
~~\text{with}~~
	v^{k+1}\coloneqq\nabla\h(x^k)-\nabla\h(x^{k+1})-\nabla\f(x^k)+\nabla\f(x^{k-1}).
\end{equation}
On the other hand, in accounting for possibly unbounded sequences, additional assumptions are needed for the condition \(\|v^{k+1}\|\leq\varepsilon\) to be met in finitely many iterations.

\begin{lemma}[termination criteria]\label{thm:termination}%
	Suppose that \cref{ass:h} holds, and let \(\seq{x^k}\) be a sequence generated by \BiFRB* in the setting of \cref{thm:SD}.
	If
	\begin{enumerator}
	\item
		either \(\varphi\) is level bounded,
	\item
		or \(\conj h\) is uniformly convex (equivalently, \(h\) is uniformly smooth),
	\end{enumerator}
	then, for \(v^{k+1}\) as in \eqref{eq:subgrad} it holds that \(v^{k+1}\to0\).
	Thus, for any \(\varepsilon>0\) the condition \(\|v^{k+1}\|\leq\varepsilon\) is satisfied for all \(k\) large enough and guarantees \(\dist(0,\hat\partial\varphi(x^{k+1}))\leq\varepsilon\).
	\begin{proof}
		The implication of \(\|v^{k+1}\|\leq\varepsilon\) and \(\varepsilon\)-stationarity of \(x^{k+1}\) has already been discussed.
		If \(\varphi\) is level bounded, then \cref{thm:subseq} implies that \(\seq{x^k}\) is bounded and \(v^{k+1}\to0\) holds by continuity of \(\nabla\h\) and \(\nabla\f\).
		In case \(\conj h\) is uniformly convex, this being equivalent to uniform smoothness of \(h\) as shown in \cite[Cor. 2.8]{aze1995uniformly}, the vanishing of \(\D_{\conj h}(\nabla h(x^k),\nabla h(x^{k+1}))=\D(x^{k+1},x^k)\) implies through \cite[Prop. 4.13(IV)]{reem2019re} that \(\|\nabla h(x^k)-\nabla h(x^{k+1})\|\to0\).
		In turn, by relative smoothness necessarily \(\|\nabla\f(x^{k+1})-\nabla\f(x^k)\|\to0\) as well, overall proving that \(v^{k+1}\to0\).
	\end{proof}
\end{lemma}

		\subsection{Global convergence}
			In this subsection, we work towards the global sequential convergence of \BiFRB. To this end, we introduce a key concept which will be useful soon.
For $\eta\in(0,\infty]$, denote by $\Psi_\eta$ the class of functions $\psi:[0,\eta)\rightarrow\R_+$ satisfying the following: (i) $\psi(t)$ is right-continuous at $t=0$ with $\psi(0)=0$; (ii) $\psi$ is strictly increasing on $[0,\eta)$. 

\begin{definition}[{\cite[Def. 5]{wang2022exact}}]\label{def: GCKL}%
	Let $\func{f}{\R^n}{\Rinf}$ be proper and lsc.
	Let $\bar x\in\dom\partial f$ and $\mu\in\R$, and let $V\subseteq\dom\partial f$ be a nonempty subset.
	\begin{enumerate}
	\item
		We say that $f$ has the pointwise generalized concave Kurdyka-\L ojasiewicz~(KL) property at $\bar x\in\dom\partial f$, if there exist a neighborhood $U\ni\bar x$, $\eta\in(0,\infty]$ and a concave $\psi\in\Psi_\eta$, such that for all $x\in U\cap[0<f-f(\bar x)<\eta]$,
		\[
			\psi'_-\bigl(f(x)-f(\bar x)\bigr)\cdot\dist\bigl(0,\partial f(x)\bigr)\geq1,
		\]
		where $\psi_-'$ denotes the left derivative.
		Moreover, $f$ is a generalized concave KL function if it has the generalized concave KL property at every $x\in\dom\partial f$.
	\item
		Suppose that $f(x)=\mu$ on $V$.
		We say $f$ has the setwise\footnote{%
			We shall omit adjectives ``pointwise'' and ``setwise'' whenever there is no ambiguity.%
		}
		generalized concave KL property on $V$ if there exist $U\supset V$, $\eta\in(0,\infty]$ and a concave $\psi\in\Psi_\eta$ such that for every $x\in U\cap[0<f-\mu<\eta]$,
		\[
			\psi'_-\bigl(f(x)-\mu\bigr)\cdot\dist\bigl(0,\partial f(x)\bigr)\geq1.
		\]
	\end{enumerate}
\end{definition}

For a subset $\Omega\subseteq\R^n$, define $(\forall \varepsilon>0)$ $\Omega_\varepsilon=\set{x\in\R^n}[\dist(x,\Omega)<\varepsilon]$. 

\begin{lemma}[{{\cite[Lem. 4.4]{wang2022exact}}}]\label{lemma: Uniformize GCKL}%
	Let $\func{f}{\R^n}{\Rinf}$ be proper lsc and let $\mu\in\R$.
	Let $\Omega\subseteq\dom\partial f$ be a nonempty compact set on which $f(x)=\mu$ for all $x\in\Omega$.
	Then the following hold:
	\begin{enumerate}
	\item
		Suppose that $f$ satisfies the pointwise generalized concave KL property at each $x\in\Omega$.
		Then there exist $\varepsilon>0,\eta\in(0,\infty]$ and $\psi\in\Psi_\eta$ such that $f$ has the setwise generalized concave KL property on $\Omega$ with respect to $U=\Omega_\varepsilon$, $\eta$ and $\psi$.
	\item \label{uniformize exact modulus}%
		Set $U=\Omega_\varepsilon$ and define $h:(0,\eta)\rightarrow\R_+$ by
		\[
			h(s)=\sup\set{\dist^{-1}\left(0,\partial f(x)\right)}[{x\in U\cap[s\leq f-\mu<\eta]}].
		\]
		Then the function $\func{\tilde\psi}{[0,\eta)}{\R_+}$ defined by \((\forall 0<t<\eta)~\tilde\psi(t)=\int_0^th(s)ds\) with \(\tilde{\psi}(0)=0\) is well defined and belongs to $\Psi_\eta$.
		The function $f$ has the setwise generalized concave KL property on $\Omega$ with respect to $U$, $\eta$ and $\tilde\psi$.
		Moreover,
		\[
			\tilde{\psi}
		{}={}
			\inf\left\{
				\psi\in\Psi_\eta
				\,\Big|\,
				\text{\begin{tabular}{@{}c@{}}%
						\(\psi\) is a concave desingularizing function\\
						of $f$ on $\Omega$ with respect to $U$ and $\eta$
				\end{tabular}}
			\right\}.
		\]
	\end{enumerate}
	We say that $\tilde{\psi}$ is the exact modulus of the setwise generalized concave KL property of \(f\) on \(\Omega\) with respect to $U$ and $\eta$.
\end{lemma}

In the remainder of the section, we will make use of the norm \(\norm*{{}\cdot{}}\) on the product space \(\R^n\times\R^n\) defined as \(\norm*{(x,y)}=\|x\|+\|y\|\).

\begin{theorem}[sequential convergence of \BiFRB]\label{thm:sequential convergence}%
	Suppose that \cref{ass:h} holds, and let $\seq{x^k}$ be a sequence generated by \BiFRB* in the setting of \cref{thm:SD}.
	Define $(\forall k\in\N)$~$z^k=(x^{k+1},x^k,x^{k-1})$ and let $\omega(z^0)$ be the set of limit points of \((z_k)_{k\in\N}\).
	Define
	\[
		(\forall \omega,x,x^-\in\R^n )~\F(\omega,x,x^-)=\M(\omega;x,x^-)+\tfrac{c}{2\gamma}\D(x,x^-)+D_{\xi}(x,x^-).
	\]
	Assume in addition the following:
	\begin{enumeratass}
	\item\label{bounded}%
		\(\varphi\) is level bounded.
	\item\label{c2 and pd}%
		$f,h$ are twice continuously differentiable and $\nabla^2h$ is positive definite.
	\item\label{GCKL}%
		$(\exists r>0)~(\exists\eta\in[0,\infty))$ $\F$ satisfies the generalized concave KL property on $\Omega=\omega(z^0)$ with respect to $U=\Omega_r$ and $\eta$.
	\end{enumeratass}
	Then \(\sum_{k=0}^\infty\norm{x^{k+1}-x^k}<\infty\) and there exists $x^\star$ with $0\in\hat\partial\varphi(x^{\star})$ such that \(x^k\to x^\star\) as \(k\to\infty\).
	To be specific, there exists $k_0\in\N$ such that for all $l\geq k_0+1$
	\begin{equation}\label{finite length}
		\sum_{k=0}^\infty\norm{x^{k+1}-x^k}
	{}\leq{}
		\sum_{k=0}^l\norm{x^{k+1}-x^k}
		{}+{}
		\tfrac{4\gamma M}{c\sigma}
		\tilde\psi\left(\LL(x^{l},x^{l-1})-\varphi^\star\right),
	\end{equation}
	where $M>0$ is some constant, $\sigma>0$ is the strong convexity modulus of $h$ on $\mathcal{B}$, the closed ball in which $\seq{x^k}$ lies,
	and $\tilde{\psi}$ is the exact modulus of the generalized concave KL property associated with $\F$ produced by \cref{uniformize exact modulus}.
\end{theorem}
\begin{proof}
	Set \((\forall k\in\N)\) $\delta_{k}=\F(z^k)$ for simplicity.
	Then $\delta_k=\LL(x^k,x^{k-1})$, $\delta_{k}\to\varphi^\star$ decreasingly and $\dist(x^k,\omega)\to0$ as $k\to\infty$ by invoking \cref{thm:subseq}.
	Assume without loss of generality that \((\forall k\in\N)\) $\delta_{k}>\varphi^\star$, otherwise we would have $(\exists k_0\in\N)$ $x^{k_0}=x^{k_0+1}$ due to \cref{thm:subseq:SD}, from which the desired result readily follows by simple induction.
	Thus $(\exists k_0\in\N)$ $(\forall k\geq k_0)$ $z^k\in\Omega_r\cap[0<\F-\varphi^\star<\eta]$.
	Appealing to \cref{thm:subseq:omega} and \cref{thm:M:eq} yields that $\F$ is constantly equal to $\varphi^\star$ on the compact set $\omega(z^0)$.
	In turn, all conditions in \cref{uniformize exact modulus} are satisfied, which implies that for $k\geq k_0$
	\begin{equation}\label{formula: KL inequality of F}
		1\leq(\tilde\psi)'_-\left(\delta_{k}-\varphi^\star\right)\dist\left(0,\partial\F(z^k)\right).
	\end{equation}
	Define $(\forall k\in\N)$
	\begin{align*}
		u^k={}
	&
		\nabla^2\h(x^k)(x^k-x^{k+1})+\nabla^2\f(x^k)(x^{k+1}-x^k)+\tfrac{c}{2\gamma}\bigl(\nabla h(x^k)-\nabla h(x^{k-1})\bigr)
	\\
	&
		+\nabla\f(x^{k-1})-\nabla\f(x^k)+\nabla\xi(x^k)-\nabla\xi(x^{k-1}),
	\\
		v^k={}
	&
		\nabla^2\f(x^{k-1})(x^k-x^{k+1})+\tfrac{c}{2\gamma}\nabla^2h(x^{k-1})(x^{k-1}-x^k)+\nabla^2\xi(x^{k-1})(x^{k-1}-x^k).
	\end{align*}
	Then $\norm{u^k}\leq(M_1+M_2)\norm{x^{k+1}-x^k}+(M_1+(\nicefrac{c}{2\gamma})M_3+M_4)\norm{x^k-x^{k-1}}$ and $\norm{v^k}\leq M_1\norm{x^{k+1}-x^k}+((\nicefrac{c}{2\gamma})M_3+M_4)\norm{x^k-x^{k-1}}$, where $M_1=\sup\|\nabla^2\f(\mathcal{B})\|$, $M_2=\sup\|\nabla^2\h(\mathcal{B})\|$, $M_3=\sup\norm{\nabla^2h(\mathcal{B})}$, and $M_4=\sup\norm{\nabla^2\xi(\mathcal{B})}$.
	Applying subdifferential calculus to $\partial\F(z_k)$ yields that
	\[
		\partial\F(z^k)
	{}={}
		\bigl(\partial\varphi(x^{k+1})+\nabla\h(x^{k+1})-\nabla\h(x^k)+\nabla\f(x^{k})-\nabla \f(x^{k-1})\bigr)\times\{u^k\}\times\{v^k\},
	\]
	which together with \cref{thm:T:partial} entails that $(0,u^k,v^k)\in\partial\F(z^k)$.
	In turn, summing the aforementioned bounds on $\norm{u^k}$ and $\norm{v^k}$ gives
	\begin{equation}\label{subd F upper bound}
		\dist\left(0,\partial\F(z^k)\right)\leq M\norm*{(x^{k+1}-x^k,x^k-x^{k-1})},
	\end{equation}
	where $M=\max\set{2M_1+M_2,M_1+(\nicefrac{c}{\gamma})M_3+2M_4}$.

	Finally, we show that $\seq{x^k}$ is convergent.
	For simplicity, define $(\forall k,l\in\N)$ $\Delta_{k,l}=\tilde\psi\left(\delta_{k}-\varphi^\star\right)-\tilde\psi\left(\delta_{l}-\varphi^\star\right)$.
	Then, combining~\eqref{formula: KL inequality of F} and~\eqref{subd F upper bound} yields
	\begin{align*}
		1
	&{}\leq{}
		M(\tilde{\psi})'_-(\delta_{k}-\varphi^\star)\norm*{(x^{k+1}-x^k,x^k-x^{k-1})}
		{}\leq{}
		\tfrac{M\Delta_{k,k+1}}{\delta_k-\delta_{k+1}}\norm*{(x^{k+1}-x^k,x^k-x^{k-1})}
	\\
	&{}\leq{}
		\frac{2\gamma M\Delta_{k,k+1}}{c\left(\D(x^{k+1},x^k)+\D(x^k,x^{k-1})\right)}\norm*{(x^{k+1}-x^k,x^k-x^{k-1})}
	\\
	&{}\leq{}
		\frac{2\gamma M\Delta_{k,k+1}\norm*{(x^{k+1}-x^k,x^k-x^{k-1})}}{c \sigma\norm{x^{k+1}-x^k}^2+c \sigma\norm{x^k-x^{k-1}}^2}\leq \frac{4\gamma M\Delta_{k,k+1}}{c \sigma\norm*{(x^{k+1}-x^k,x^k-x^{k-1})}},
	\end{align*}
	where the second inequality is implied by concavity of $\tilde\psi$, the third one follows from \eqref{SD inequality}, and the fourth one holds because $\sigma>0$ is the strong convexity modulus of $h$ on $\mathcal{B}$.
	Hence,
	\begin{equation}\label{consecutive gap}
		\norm*{(x^{k+1}-x^k,x^k-x^{k-1})}\leq\tfrac{4\gamma M}{c\sigma}\Delta_{k,k+1}.
	\end{equation} 
	Summing~\eqref{consecutive gap} from $k=k_0$ to an arbitrary $l\geq k_0+1$ and passing $l$ to infinity justifies~\eqref{finite length}.
	A similar procedure shows that $\seq{x^k}$ is Cauchy, which together with \cref{thm:subseq:omega} entails the rest of the statement. 
\end{proof}

\begin{remark}
	Note that $\tilde\psi$ is the smallest concave desingularizing function associated with $\F$.
	Therefore~\eqref{finite length} is the sharpest upper bound on $\sum_{k=0}^\infty\norm{x^{k+1}-x^k}$ produced by the usual KL convergence framework.
	We refer readers to~\cite[\S6]{bolte2018first} for a summary of such a framework.
\end{remark}

The corollary below states that semialgebraic functions satisfy the assumptions in \cref{thm:sequential convergence}. 

\begin{corollary}
	Suppose that \cref{ass:h} holds, and let $\seq{x^k}$ be a sequence generated by \BiFRB* in the setting of \cref{thm:SD}.
	Assume in addition that%
	\begin{enumeratass}
	\item
	$\varphi$ is level-bounded,
	\item
		$f,h$ are twice continuously differentiable and $\nabla^2h$ is positive definite, and
	\item
		$\varphi,h$ are semialgebraic.
	\end{enumeratass}
	Then \(\sum_{k=0}^\infty\norm{x^{k+1}-x^k}<\infty\) and there exists $x^\star$ with $0\in\hat\partial\varphi(x^{\star})$ such that \(x^k\to x^\star\).%
\end{corollary}
\begin{proof}
	\Cref{thm:subseq:bounded} entails that $\seq{x_k}$ is bounded.
	Note that the class of semialgebraic functions is closed under summation and satisfies the generalized concave KL property; see, e.g.,~\cite[\S4.3]{attouch2010proximal} and~\cite[\S2.2]{wang2022exact}.
	Hence so does $\F$ defined in \cref{thm:sequential convergence}.
	Applying \cref{thm:sequential convergence} completes the proof.
\end{proof}

		\subsection{Convergence rates}
			Having established convergence of \BiFRB, we now turn to its rate.
Recall that a function is said to have KL exponent $\theta\in[0,1)$ if it satisfies the generalized concave KL property (recall \cref{def: GCKL}) and there exists a desingularizing function of the form $\psi(t)=ct^{1-\theta}$ for some $c>0$.

\begin{theorem}[function value and sequential convergence rate]\label{thm: rate}%
	Suppose that all the assumptions in \cref{thm:sequential convergence} are satisfied, and follow the notation therein.
	Define $(\forall k\in\N)$ $e_k=\LL(x^{k+1},x^{k})-\varphi^\star$.
	Assume in addition that $\F$ has KL exponent $\theta\in[0,1)$ at $(x^\star,x^\star,x^\star)$.
	Then the following hold:
	\begin{enumerate}
	\item
		If $\theta=0$, then $e_k\to0$ and $x^k\to x^\star$ after finite steps.
	\item
		If $\theta\in(0,\nicefrac{1}{2}]$, then there exist $c_1,\hat c_1>0$ and $Q_1,\hat Q_1\in[0,1)$ such that for $k$ sufficiently large,
		\[
			e_k\leq\hat c_1\hat Q_1^k\text{ and }\norm{x^k-x^\star}\leq c_1Q_1^k.
		\]
	\item
		If $\theta\in(\nicefrac{1}{2},1)$, then there exist $c_2,\hat c_2>0$ such that for $k$ sufficiently large,
		\[
			e_k\leq \hat c_2k^{-\frac{1}{2\theta-1}} \text{ and } \norm{x^k-x^\star}\leq c_2k^{-\frac{1-\theta}{2\theta-1}}.
		\]
	\end{enumerate}
\end{theorem}
\begin{proof}
	Assume without loss of generality that $\F$ has a desingularizing function $\psi(t)=\nicefrac{t^{1-\theta}}{(1-\theta)}$ and let $(\forall k\in\N)$ $\delta_{k}=\sum_{i=k}^\infty\norm{x^{i+1}-x^i}$.
	We claim that
	\begin{equation}\label{formula: function value rate to sequential}
		(\forall k\geq k_0)	~\delta_{k}\leq\tfrac{4\gamma M}{(1-\theta)c\sigma}e_{k-1}^{1-\theta}+2\sqrt{\tfrac{2\gamma}{c\sigma}}\sqrt{e_{k-1}},
	\end{equation}
	which will be justified at the end of this proof. It is routine to see that the desired sequential rate can be implied by those of $(e_k)$ through~\eqref{formula: function value rate to sequential}; see, e.g.,~\cite[Thm. 5.3]{wang2022bregman}, therefore it suffices to prove convergence rate of $(e_k)$.

	Recall from \cref{thm:subseq:SD} that $(e_k)$ is a decreasing sequence converging to $0$.
	Then invoking the KL exponent assumption yields $e_{k-1}^\theta=[\F(x^{k+1},x^k,x^{k-1})-\varphi^\star]^\theta\leq \dist(0,\partial\F(x^{k+1},x^k,x^{k-1}))$, which together with~\eqref{subd F upper bound} implies that
	\begin{equation}\label{formula: rate-2}
		e_{k}^\theta\leq e_{k-1}^\theta\leq M\norm*{(x^{k+1}-x^k,x^k-x^{k-1})}.
	\end{equation}
	Appealing again to \cref{thm:subseq:SD} gives 
	\begin{align*}
		e_{k-1}-e_k&=\LL(x^k,x^{k-1})-\LL(x^{k+1},x^{k})\\
		&\geq\tfrac{c}{2\gamma}\left[\D(x^{k+1},x^k)+\D(x^k,x^{k-1})\right]\\
		&\geq\tfrac{c\sigma}{8\gamma}\norm*{(x^{k+1}-x^k,x^k-x^{k-1})}^2\\
		&\geq\tfrac{c\sigma}{8\gamma M^2 }e_k^{2\theta},
	\end{align*}
	where the last inequality is implied by~\eqref{formula: rate-2}.
	Then applying~\cite[Lem. 10]{bot2020proximal} justifies the desired convergence rate of $(e_k)$. 

	Finally, we show that~\eqref{formula: function value rate to sequential} holds.
	Invoking \cref{c2 and pd,thm:subseq:SD} entails
	\(
		\norm*{(x^{k+1}-x^k,x^k-x^{k-1})}^2\leq(\nicefrac{4}{\sigma})[\D(x^{k+1},x^k)+\D(x^k,x^{k-1})]\leq(\nicefrac{8\gamma}{c\sigma})(e_{k-1}-e_k)\leq(\nicefrac{8\gamma}{c\sigma})e_{k-1},
	\)
	thus $\norm{x^k-x^{k-1}}\leq2\sqrt{(\nicefrac{2\gamma}{c\sigma})e_{k-1}}$.
	In turn,
	\begin{align*}
		\delta_{k}&\leq\delta_{k}+\norm{x^k-x^{k-1}}\leq\tfrac{4\gamma M}{c\sigma}\tilde{\psi}(e_k)+2\sqrt{\tfrac{2\gamma}{c\sigma}}\sqrt{e_{k-1}}\\
		&\leq \tfrac{4\gamma M}{c\sigma}\tilde{\psi}(e_{k-1})+2\sqrt{\tfrac{2\gamma}{c\sigma}}\sqrt{e_{k-1}}\leq\tfrac{4\gamma M}{(1-\theta)c\sigma}e_{k-1}^{1-\theta}+2\sqrt{\tfrac{2\gamma}{c\sigma}}\sqrt{e_{k-1}},
	\end{align*}
	where the second inequality follows from~\eqref{finite length}, the third one holds due to the fact that $e_k\leq e_{k-1}$ and the monotonicity of $\tilde\psi$, and the last one is implied by the KL exponent assumption and Lemma~\ref{lemma: Uniformize GCKL}, as claimed.
\end{proof}

	\section{Globalizing fast local methods with \texorpdfstring{$i^*$}{i*}FRB}\label{sec:CLyD}%
		With some due care, the method can be enhanced in the context of the continuous-Lyapunov descent (CLyD) framework \cite[\S4]{themelis2018proximal}, so as to globalize fast local methods \(x^+=x+d\) by using the same oracle as \BiFRB.
Methods of quasi-Newton type constitute a convenient class of candidate local methods.
A prototypical use case hinges on the inclusion \(x\in\operatorname S(x)\coloneqq\T(x,x)\) encoding necessary optimality conditions, cf. \cref{thm:T:fix}, so that fast update directions \(d\) can be retrieved based on the root-finding problem \(0\in(\id-\operatorname S)(x)\); see, \eg, \cite[§7]{facchinei2003finite} and \cite{izmailov2014newton}.
Regardless, while the globalization framework is flexible to accommodate \emph{any} update direction and yet retains convergence of \BiFRB, it promotes the ones triggering fast local convergence as it will be demonstrated in \cref{thm:CLyD:tau1}.

\begin{algorithm}[hbt]
	\caption{%
		\protect\BiFRB\ linesearch extension%
	}%
	\label{alg:CLyD}%
	\begin{algorithmic}[1]%
\setlength{\itemsep}{0.75ex}%
\item[%
	Choose
	\(x^{-1},x^0\in C\),~
	\(\beta,\gamma\) as in \cref{thm:SD},~
	\(\delta\in(0,1)\)%
]%
\item[%
	Set \(y^{-1}=x^{-1}\), and let \(c>0\) and \(\LL\) be as in \cref{thm:SD}%
]%
\item[%
	Iterate for \(k=0,1,\ldots\) until a termination criterion is met (cf. \cref{thm:CLyD:finite})%
]%
	\State\label{state:CLyD:barx}%
		Compute \(\bar x^k\in\T(x^k,y^{k-1})\)
	\Comment{%
		\smash{%
			\begin{tabular}[t]{@{}r@{}r@{}}%
				& \(\LL(\bar x^k,x^k)\leq\LL(x^k,y^{k-1})\)\\
				& \({}-\frac{c}{2\gamma}\D(\bar x^k,x^k)-\frac{c}{2\gamma}\D(x^k,y^{k-1})\)%
			\end{tabular}%
		}%
	}%
	\State\vspace*{3pt}%
		Choose an update direction \(d^k\) at \(x^k\)
	\State\label{state:CLyD:x+}%
		\begin{tabular}[t]{@{}l@{}}
			Set
			~\(y^k=x^k+\tau_k d^k\)~
			and
			~\(x^{k+1}=(1-\tau_k)\bar x^k+\tau_k(x^k+d^k)\)
		\\
			with \(\tau_k\) the largest in
			\(\set{1,\nicefrac12,\dots}\) such that
		\end{tabular}
		\begin{equation}\label{eq:LS}
			\LL(x^{k+1},y^k)
		{}\leq{}
			\LL(x^k,y^{k-1})
			{}-{}
			\tfrac{\delta c}{2\gamma}\left(
				\D(\bar x^k,x^k)
				{}+{}
				\D(x^k,y^{k-1})
			\right)
		\end{equation}
\end{algorithmic}
\end{algorithm}

The algorithmic framework revolves around two key facts: continuity of \(\LL\) and its decrease after \BiFRB\ steps \((x,y^-)\mapsto (\bar x,x)\) with \(\bar x\in\T(x,y^-)\).
(The reason for introducing an auxiliary variable \(y\) will be discussed after \cref{thm:CLyD:subseq}.)
Thus, not only is \(\LL\) smaller than \(\LL(x,y^-)\) at \((\bar x,x)\), but also at sufficiently close points, thereby ensuring that by gradually pushing the tentative fast update towards the safeguard \((\bar x,x)\) a decrease on \(\LL\) is eventually achieved.
This fact is formalized next.

\begin{theorem}[well definedness and asymptotic analysis of \cref{alg:CLyD}]\label{thm:CLyD:subseq}%
	Suppose that \cref{ass:basic} and the bounds on \(\gamma\) and \(\beta\) as in \cref{thm:SD} are satisfied.
	Then, the following hold for the iterates generated by \cref{alg:CLyD}:
	\begin{enumerate}
	\item\label{thm:CLyD:LS}%
		Regardless of what the selected update direction \(d^k\) is, the linesearch at \cref{state:CLyD:x+} always succeeds:
		either \(\bar x^k=x^k=y^{k-1}\) holds, in which case \(0\in\hat\partial\varphi(\bar x^k)\), or there exists \(\bar\tau_k>0\) such that \eqref{eq:LS} holds for every \(\tau_k\leq\bar\tau_k\).
	\item\label{thm:CLyD:SD}%
		\(\sum_{k\in\N}\D(\bar x^k,x^k)<\infty\), and in particular \(\D(\bar x^k,x^k)\to0\).
	\item\label{thm:CLyD:bounded}%
		If \(\cost\) is coercive, then the sequence \(\seq{x^k}\) remains bounded.
	\item\label{thm:CLyD:omega}%
		Suppose that \cref{ass:h} holds, and let \(\omega\) be the set of limit points of \(\seq{x^k}\).
		Then, \(\varphi\) is constant on \(\omega\) with value \(\varphi^\star\), and for every \(x^\star\in\omega\) it holds that \(x^\star\in\T(x^\star,x^\star)\) and \(0\in\hat\partial\varphi(x^\star)\).
	\item\label{thm:CLyD:finite}%
		Under the assumptions of \cref{thm:termination}, for any \(\varepsilon>0\) the condition \(\|\nabla\h(x^k)-\nabla\h(\bar x^k)-\nabla\f(x^k)+\nabla\f(y^{k-1})\|\leq\varepsilon\) holds for all \(k\) large enough and guarantees that \(\dist(0,\hat\partial\varphi(\bar x^k))\leq\varepsilon\).
	\end{enumerate}
	\begin{proof}~
		\begin{proofitemize}
		\item\ref{thm:CLyD:LS}~
			Stationarity of \(\bar x^k\) when \(\bar x^k=x^k=y^k\) follows from \cref{thm:T:fix}.
			Otherwise, let \(x_\tau^{k+1}\coloneqq(1-\tau)\bar x^k+\tau(x^k+d^k)\) and \(y_\tau^k\coloneqq x^k+\tau d^k\), and note that \((x_\tau^{k+1},y_\tau^k)\to(\bar x^k,x^k)\) as \(\tau\searrow0\).
			By continuity of \(\LL\) we thus have
			\begin{align*}
				\lim_{\tau\searrow0}\LL(x_\tau^{k+1},y_\tau^k)
			{}={} &
				\LL(\bar x^k,x^k)
			\\
			{}\overrel*[\leq]{\ref{thm:SD}}{} &
				\LL(x^k,y^{k-1})
				{}-{}
				\tfrac{c}{2\gamma}\left(
					\D(\bar x^k,x^k)
					{}+{}
					\D(x^k,y^{k-1})
				\right)
			\shortintertext{%
				and since one at least among \(\D(\bar x^k,x^k)\) and \(\D(x^k,y^{k-1})\) is strictly positive, \(\delta\in(0,1)\), and \(c>0\),
			}
			{}<{} &
				\LL(x^k,y^{k-1})
				{}-{}
				\tfrac{\delta c}{2\gamma}\left(
					\D(\bar x^k,x^k)
					{}+{}
					\D(x^k,y^{k-1})
				\right).
			\end{align*}
			It then follows that there exists \(\bar\tau_k>0\) such that
			\[
				\LL(x_\tau^{k+1},y_\tau^k)
			{}\leq{}
				\LL(x^k,y^{k-1})
				{}-{}
				\tfrac{\delta c}{2\gamma}\left(
					\D(\bar x^k,x^k)
					{}+{}
					\D(x^k,y^{k-1})
				\right)
			\]
			holds for every \(\tau\in(0,\bar\tau_k]\), which is the claim.
		\item\ref{thm:CLyD:SD} \& \ref{thm:CLyD:bounded}~
			Follow by telescoping \eqref{eq:LS}, since \(\inf\LL=\inf\varphi>-\infty\) by \cref{thm:SD:inf} and from the fact that \(\LL(x^k,y^{k-1})\leq\LL(x^0,y^{-1})\).
		\item\ref{thm:CLyD:omega} \& \ref{thm:CLyD:finite}~
			Having shown the validity of \eqref{eq:LS}, the proof follows from the same arguments of those of \cref{thm:subseq:omega,thm:termination}.
		\qedhere
		\end{proofitemize}
	\end{proof}
\end{theorem}

It should be noted that
\cref{thm:CLyD:subseq} remains valid even with the simpler choice \(y^k=x^k\).
The apparently wasteful choice of \(y^k\) as in \cref{alg:CLyD} is instead of paramount importance
in promoting ``good'' directions by enabling acceptance of unit stepsize.
This is in sharp contrast with known issues of other nonsmooth globalization strategies, where
convergence severely hampered (\emph{Maratos' effect} \cite{maratos1978exact}, see also \cite[\S6.2]{izmailov2014newton}).
A quality assessment on the direction is provided by the following notion.

\begin{definition}[superlinearly convergent directions {\cite[Eq. (7.5.2)]{facchinei2003finite}}]%
	Relative to a sequence \(\seq{x^k}\) converging to a point \(x^\star\) we say that \(\seq{d^k}\) are superlinearly convergent directions if
	\[
		\lim_{k\to\infty}\frac{\|x^k+d^k-x^\star\|}{\|x^k-x^\star\|}=0.
	\]
\end{definition}

\begin{theorem}[acceptance of unit stepsize]\label{thm:CLyD:tau1}%
	Suppose that \cref{ass:basic,ass:h} and the bounds on \(\gamma\) and \(\beta\) as in \cref{thm:SD} are satisfied.
	Suppose further that \(f,h\in\C^2\) with \(\nabla h\succ0\), and that \(\seq{x^k}\) converges to a strong local minimum \(x^\star\) for \(\varphi\) satisfying \(\T(x^\star,x^\star)=\set{x^\star}\).\footnote{%
		Although it is only the inclusion \(x^\star\in\T(x^\star,x^\star)\) which is guaranteed for any limit point in the full generality of \cref{ass:basic,ass:h} (cf. \cref{thm:CLyD:omega}), additionally requiring single-valuedness is a negligible extra assumption as it can be inferred from \cref{thm:T:fix}; see \cite[\S3.1]{ahookhosh2021bregman} and \cite[\S2.4]{themelis2018proximal} for a detailed discussion.%
	}
	If the directions \(\seq{d^k}\) are superlinearly convergent with respect to \(\seq{x^k}\), then eventually \(\tau_k=1\) is always accepted at \cref{state:CLyD:x+}, and the algorithm reduces to the local method \(x^{k+1}=x^k+d^k\) and converges at superlinear rate.
	\begin{proof}
		Let \(\varphi^\star=\varphi(x^\star)\) be the limit point of \(\seq{\LL(x^{k+1},y^k)}\).
		Since \(x^k\to x^\star\), \(\bar x^k\to x^\star\) as well, and strong local minimality of \(x^\star\) implies that \(\varphi(\bar x^k)-\varphi^\star\geq\frac{\sigma}{2}\|\bar x^k-x^\star\|^2\) holds for some \(\sigma>0\) and all \(k\) large enough.
		In turn, for \(k\) large it holds that
		\begin{align*}
			\LL(x^k,y^{k-1})
		{}\overrel[\geq]{\eqref{eq:L:geq}}{} &
			\varphi(\bar x^k)
			{}+{}
			\tfrac{c}{\gamma}\D(\bar x^k,x^k)
			{}+{}
			\tfrac{c}{2\gamma}\D(x^k,y^{k-1})
		\\
		{}\geq{} &
			\varphi^\star
			{}+{}
			\tfrac{\sigma}{2}\|\bar x^k-x^\star\|^2
			{}+{}
			\tfrac r2\|\bar x^k-x^k\|^2
		\shortintertext{%
			for some \(r>0\) (since \(\nabla^2h\succ0\)), and by Young's inequality
		}
		{}\geq{} &
			\varphi^\star
			{}+{}
			\tfrac\sigma2\|\bar x^k-x^\star\|^2
			{}+{}
			\tfrac r2\left(
				\tfrac{1}{1+\epsilon}\|x^k-x^\star\|^2
				{}-{}
				\tfrac{1}{\epsilon}\|\bar x^k-x^\star\|^2
			\right)
		\end{align*}
		for every \(\epsilon>0\).
		By choosing \(\epsilon=\nicefrac r\sigma\) one obtains that
		\begin{equation}\label{eq:LL:strlocmin}
			\LL(x^k,y^{k-1})-\varphi^\star\geq\tfrac\mu2\|x^k-x^\star\|^2
		\quad
			\text{for \(k\) large enough,}
		\end{equation}
		where \(\mu\coloneqq\nicefrac{\sigma r}{(\sigma+r)}>0\). On the other hand, since \(\LL(x,x)=\env(x,x)\) for any \(x\), by definition of \(\env\) (cf. \eqref{eq:env} with \(x^-=x\)) it follows that
		\begin{align*}
			\LL(x^k+d^k,x^k+d^k)
		{}={} &
			\env(x^k+d^k,x^k+d^k)
		\numberthis\label{eq:CLyD:LL=env}
		\\
		{}\leq{} &
			\M(x^\star;x^k+d^k,x^k+d^k)
		{}={}
			\varphi^\star+\D_{\h}(x^\star,x^k+d^k)
		\\
		{}\leq{} &
			\varphi^\star+\tfrac{L}{2}\|x^k+d^k-x^\star\|^2
		\end{align*}
		holds for some \(L>0\) and all \(k\) large enough, where the last inequality uses the fact that \(f,h\in\C^2\) (hence \(\h\in\C^2\) too).
		Combined with \eqref{eq:LL:strlocmin} and superlinearity of the directions \(\seq{d^k}\) we obtain that
		\begin{equation}\label{eq:epsk}
			\varepsilon_k
		{}\coloneqq{}
			\frac{
				\LL(x^k+d^k,x^k+d^k)-\varphi^\star
			}{
				\LL(x^k,y^{k-1})-\varphi^\star
			}
		{}\to{}
			0
		\quad
			\text{as \(k\to\infty\).}
		\end{equation}
		Observe that for any \(z^k\in\T(\bar x^k,x^k)\) eventually it holds that \(\LL(\bar x^k,x^k)\geq\varphi(z^k)\geq\varphi^\star\).
		In fact, the first inequality owes to \eqref{eq:L:geq}, and the second one holds for \(k\) large enough because of local minimality of \(x^\star\) for \(\varphi\) and the fact that
	\ifsiam
		\(
	\else
		\[
	\fi
			\lim_{k\to\infty}z^k
		{}\in{}
			\limsup_{k\to\infty}\T(\bar x^k,x^k)
		{}\subseteq{}
			\T(x^\star,x^\star)
		{}={}
			\set*{x^\star}
	\ifsiam
		\)
	\else
		\]
	\fi
		by osc of \(\T\) (cf. \cref{thm:T:osc}).
		Then, for \(k\) large enough so that \(\varepsilon_k\leq1\) also holds, we have
		\begin{align*}
			\LL(x^k+d^k,x^k+d^k)-\LL(x^k,y^{k-1})
		{}={} &
			(\varepsilon_k-1)\left(
				\LL(x^k,y^{k-1})-\varphi^\star
			\right)
		\\
		{}\leq{} &
			(\varepsilon_k-1)\left(
				\LL(x^k,y^{k-1})-\varphi(z^k)
			\right)
		\shortintertext{%
			for \(\bar x^k\in\T(x^k,y^{k-1})\), so that \cref{thm:SD} yields
		}
		{}\leq{} &
			\tfrac{(\varepsilon_k-1)c}{2\gamma}\left(\D(\bar x^k,x^k)+\D(x^k,y^{k-1})\right).
		\end{align*}
		Since \(\varepsilon_k\to0\) and \(\delta\in(0,1)\), eventually \(\varepsilon_k-1\leq-\delta\), hence
		\[
			\LL(x^k+d^k,x^k+d^k)
		{}\leq{}
			\LL(x^k,y^{k-1})
			{}-{}
			\tfrac{\delta c}{2\gamma}\D(\bar x^k,x^k)
			{}-{}
			\tfrac{\delta c}{2\gamma}\D(x^k,y^{k-1}),
		\]
		implying that the first attempt with \(\tau_k=1\) passes the linesearch condition (since, when \(\tau_k=1\), one has \(x^{k+1}=y^k=x^k+d^k\)).
		In particular, the sequence \(\seq{x^k}\) eventually reduces to \(x^{k+1}=x^k+d^k\), and thus converges superlinearly.
	\end{proof}
\end{theorem}

	The core of the proof hinges on showing that \(\varepsilon_k\) as in \eqref{eq:epsk} vanishes.
	The same argument cannot be achieved by a naive linesearch prescribing \(y^k=x^k\).
	Indeed,
	\begin{Align*}
		0
	{}\leq{} &
		\env(x^k+d^k,x^k)
		{}-{}
		\varphi(x^\star)
	\\
	{}\leq{} &
		\underbracket[0.5pt]{
			\D_{\h}(x^\star,x^k+d^k)
		}_{O(\|x^k+d^k-x^\star\|^2)}
		{}+{}
		\innprod*{x^\star-x^k-d^k}{
			\underbracket[0.5pt]{
				\nabla\f(x^k+d^k)-\nabla\f(x^k)
			}_{O(\|x^k-x^\star\|)}
		},
	\end{Align*}
	where the first inequality owes to local minimality of \(x^\star\) and the fact that \(x^k+d^k\to x^\star\), the second one from the expression \eqref{eq:M-innprod} of \(\M\) and the definition of the envelope, cf. \eqref{eq:env}, and the big-\(O\) estimates from local smoothness of \(h\) and \(f\).
	In particular,
	\(
		\env(x^k+d^k,x^k)
		{}-{}
		\varphi(x^\star)
	{}={}
		o(\|x^k-x^\star\|^2)
	\).
	Observing that \eqref{eq:LL:strlocmin} is still valid, denoting \(\hat\xi\coloneqq\xi+(\nicefrac{c}{2\gamma})h\) so that \(\LL=\env+\D_{\hat\xi}\) one then has
	\[
		\varepsilon_k
	{}\approx{}
		\frac{
			\D_{\hat\xi}(x^k+d^k,x^k)
		}{
			\LL(x^k,x^{k-1})-\varphi^\star
		}
	{}\approx{}
		\frac{
			c'\|d^k\|^2
		}{
			\LL(x^k,x^{k-1})-\varphi^\star
		}
	{}\approx{}
		\frac{
			c''\|x^k-x^\star\|^2
		}{
			\LL(x^k,x^{k-1})-\varphi^\star
		},
	\]
	where ``\(\approx\)'' denotes equality up to vanishing terms and \(c',c''>0\) are some constants due to \cite[Lem. 7.5.7]{facchinei2003finite}.
	In this process, the information of superlinearity of \(\seq{d^k}\) is lost, and the vanishing of \(\varepsilon_k\) cannot be established.
	Instead, as is apparent from \eqref{eq:CLyD:LL=env} the choice of \(y^k\) as in \cref{state:CLyD:x+} guarantees that on the first trial step \(\LL\) coincides with a Bregman Moreau envelope, a function more tightly connected to the cost \(\varphi\).%

	\section{Conclusions}\label{sec:conclusions}%
		This work contributes a mirror inertial forward-reflected-back\-ward splitting algorithm (\BiFRB) and its linesearch enhancement, extending the forward-reflected-back\-ward method proposed in~\cite{malitsky2020forward} to the nonconvex and relative smooth setting.
We have shown that the proposed algorithms enjoy pleasant properties akin to other splitting methods in the same setting.
However, our methodology deviates from tradition through the \BiFRB-envelope, an envelope function defined on  a product space that takes inertial terms into account, which, to the best of our knowledge, is the first of its kind and thus could be instrumental for future research.
This approach also requires the inertial parameter to be negative, which coincides with a recent result~\cite{dragomir2022optimal} regarding the impossibility of accelerated non-Euclidean algorithms under relative smoothness.
Thus, it would be tempting to see whether an explicit example can be constructed to prove the sharpness of such restrictive assumption.
It is also worth applying our technique to other two-stage splitting methods, such as Tseng's method, to obtain similar extensions.

	\begin{appendix}
		\phantomsection
		\addcontentsline{toc}{section}{Appendix}%
		\section{Proof of \texorpdfstring{\cref{thm:SD}}{Theorem \ref*{thm:SD}}}\label{sec:SD:proof}%
			\subsection{Convex \texorpdfstring{\(\f\)}{f\_beta} case}\label{sec:fbconv}%
				\begin{lemma}\label{thm:fconv:general}%
	Suppose that \cref{ass:basic} holds and let \(\gamma>0\) and \(\beta\in\R\) be such that \(\f\coloneqq f-\frac\beta\gamma h\) is a convex function.
	Then, for every \(x,x^-\in C\) and \(\bar x\in\T(x,x^-)\)
	\begin{align}\label{eq:decrease:fconv}
		\left(\env+\D_{\f}\right)(\bar x,x)
	{}\leq{} &
		\left(\env+\D_{\f}\right)(x,x^-)
		{}-{}
		\D_{\h-2\f}(\bar x,x).
	\end{align}
	If \(\h-\f\) too is convex, then \(\env+\D_{\f}\) has the same infimum of \(\cost\), and is level bounded iff so is \(\cost\).%
	\begin{proof}
		All the claimed inequalities follows from \eqref{eq:descent} together with the fact that \(\D_{\f}\geq0\), and that \(\D_{\h-\f}\geq0\) too when \(\h-\f\) is convex.
		When both \(\f\) and \(\h-\f\) are convex, \cref{thm:env:eq} implies that
		\begin{equation}\label{eq:fconv:geq}
			\env(y,y^-)+\D_{\f}(y,y^-)\geq\varphi(\bar y)\geq\inf\cost
		\quad
			\forall y,y^-\in C,\ \bar y\in\T(y,y^-),
		\end{equation}
		proving that \(\inf(\env+\D_{\f})\geq\inf\cost\).
		The converse inequality follows from \cref{thm:env:leq} by observing that \((\env+\D_{\f})(x,x)=\env(x,x)\leq\varphi(x)\).

		To conclude, suppose that \(\env+\D_{\f}\) is not level bounded, and consider an unbounded sequence \(\seq{x_k,x_k^-}\) such that \(\env(x_k,x_k^-)+\D_{\f}(x_k,x_k^-)\leq\ell\), for some \(\ell\in\R\).
		Then, it follows from \eqref{eq:fconv:geq} that \(\varphi(\bar x_k)\leq\ell\), where \(\bar x_k\in\T(x_k,x_k^-)\).
		From local boundedness of \(\T\) (\cref{thm:T:osc}) it follows that \(\seq{\bar x_k}\) too is unbounded, showing that \(\varphi\) is not level bounded either.
		The converse holds by observing that \((\env+\D_{\f})(x,x)=\env(x,x)\leq\cost(x)\), cf. \cref{thm:env:leq}.
	\end{proof}
\end{lemma}

In the setting of \cref{thm:SD:fconv}, inequality \eqref{eq:decrease:fconv} can equivalently be written in terms of \(\LL\) as
\begin{align*}
	\LL(\bar x,x)
{}\leq{} &
	\LL(x,x^-)
	{}-{}
	\D_{\h-2\f-\frac c\gamma h}(\bar x,x)
	{}-{}
	\tfrac{c}{2\gamma}\D(\bar x,x)
	{}-{}
	\tfrac{c}{2\gamma}\D(x,x^-)
\\
{}\leq{} &
	\LL(x,x^-)
	{}-{}
	\tfrac{c}{2\gamma}\D(\bar x,x)
	{}-{}
	\tfrac{c}{2\gamma}\D(x,x^-),
\end{align*}
where the second inequality owes to the fact that \(\D_{\h-2\f-\frac c\gamma h}\geq0\), since \(\h-2\f-\frac c\gamma h\) is convex, having
\[
	\h-2\f-\tfrac c\gamma h
{}={}
	\tfrac{1+2\beta-c}{\gamma}h-3f
{}={}
	\overbracket*{\tfrac{1+2\beta+3\alpha p_{-f,h}-c}{\gamma}}^{=0}h+2(-f-\sigma_{-f,h}h),
\]
the coefficient of \(h\) being null by definition of \(c\), and \(-f-\sigma_{-f,h}h\) being convex by definition of the relative weak convexity modulus \(\sigma_{-f,h}\), cf. \cref{def:relweak}.
This proves \eqref{eq:SD}; inequality \eqref{eq:L:geq} follows similarly by observing that
\[
	0
{}\leq{}
	\D_{\h-2\f-\frac c\gamma h}
{}={}
	\D_{\h-\f}-\D_{\f}-\tfrac c\gamma\D
{}\leq{}
	\D_{\h-\f}-\tfrac c\gamma\D,
\]
so that
\[
	\LL(x,x^-)
{}\overrel{\ref{thm:env:eq}}{}
	\varphi(\bar x)
	{}+{}
	\overbracket*{
		\D_{\h-\f}
		\vphantom{x^-}
	}^{\geq\frac c\gamma\D}(\bar x,x)
	{}+{}
	\overbracket*{
		\D_{\f}(\bar x,x^-)
	}^{\geq0}
	{}+{}
	\tfrac{c}{2\gamma}\D(x,x^-).
\]
In turn, \cref{thm:SD:inf} follows from the same arguments as in the proof of \cref{thm:fconv:general}.

			\subsection{Lipschitz differentiable \texorpdfstring{\(\f\)}{f\_beta} case}\label{sec:fbsmooth}%
				\begin{lemma}\label{thm:fsmooth:general}%
	Additionally to \cref{ass:basic}, suppose that \(\f\) is \(L_{\f}\)-Lipschitz differentiable for some \(L_{\f}\geq0\).
	Then, for every \(x,x^-\in C\) and \(\bar x\in\T(x,x^-)\)
	\begin{equation}\label{eq:decrease:fsmooth}
		\left(\env+\D_{L_{\f}\j}\right)(\bar x,x)
	{}\leq{}
		\left(\env+\D_{L_{\f}\j}\right)(x,x^-)
		{}-{}
		\D_{\h-2L_{\f}\j}(\bar x,x).
	\end{equation}
	If \(\h-L_{\f}\j\) is a convex function, then \(\env+\D_{L_{\f}\j}\) has the same infimum of \(\cost\), and is level bounded iff so is \(\cost\).
	\begin{proof}
		By means of the three-point identity, that is, by using \eqref{eq:M-innprod} in place of \eqref{eq:M-D}, inequality \eqref{eq:descent} can equivalently be written as
		\begin{align*}
			\env(\bar x,x)
		{}\leq{}
			\cost(\bar x)
		{}={} &
			\env(x,x^-)
			{}-{}
			\D_{\h}(\bar x,x)
			{}-{}
			\innprod*{\bar x-x}{\nabla\f(x)-\nabla\f(x^-)}
		\shortintertext{%
			which by using Young's inequality on the inner product and \(L_{\f}\)-Lipschitz differentiability yields
		}
		{}\leq{} &
			\env(x,x^-)
			{}-{}
			\D_{\h}(\bar x,x)
			{}+{}
			\tfrac{L_{\f}}{2}\|\bar x-x\|^2
			{}+{}
			\tfrac{L_{\f}}{2}\|x-x^-\|^2.
		\end{align*}
		Rearranging and using the fact that \(\D_{\j}(x,y)=\frac12\|x-y\|^2\) yields the claimed inequality.
		
		Under convexity of \(\h-L_{\f}\j\), we may draw the same conclusions as in the proof of \cref{thm:fconv:general} by observing that the same application of Young's inequality above, combined with the expression \eqref{eq:M-innprod} of \(\M\), gives
		\[
			\env(y,y^-)
		{}={}
			\M(\bar y;y,y^-)
		{}\geq{}
			\varphi(\bar y)
			{}+{}
			\D_{\h}(\bar y,y)
			{}-{}
			\tfrac{L_{\f}}{2}\|\bar y-y\|^2
			{}-{}
			\tfrac{L_{\f}}{2}\|y-y^-\|^2,
		\]
		holding for any \(y,y^-\in C\) and \(\bar y\in\T(y,y^-)\).
	\end{proof}
\end{lemma}

We will pattern the arguments in the previous section, and observe that inequality \eqref{eq:decrease:fsmooth} can equivalently be written in terms of \(\LL\) as
\[
	\LL(\bar x,x)
{}\leq{}
	\LL(x,x^-)
	{}-{}
	\overbracket*{
		\D_{\h-2L_{\f}\j-\frac c\gamma h}
	}^{\geq0}(\bar x,x)
	{}-{}
	\tfrac{c}{2\gamma}\D(\bar x,x)
	{}-{}
	\tfrac{c}{2\gamma}\D(x,x^-).
\]
Once again, the fact that \(\D_{\h-2L_{\f}\j-\frac c\gamma h}\geq0\) owes to convexity of \(\h-2L{\f}\j-\frac c\gamma h\), having
{%
\ifsiam
	\mathtight[0.8]%
\fi
	\begin{align*}
		\h-2L_{\f}\j-\tfrac c\gamma h
	{}={} &
		\tfrac{1-c}{\gamma}h
		{}-{}
		f
		{}-{}
		2L_{\f}\j
	\\
	{}={} &
		\def\myvphantom{\vphantom{\tfrac{1+\gamma\sigma_{-f,h}-c}{\gamma}}}
		\overbracket[0.5pt]{
			\myvphantom
			\tfrac{1+\gamma\sigma_{-f,h}-c}{\gamma}
		}^{=\nicefrac{2L_{\f}}{\sigma_h}\geq0}
		(\overbracket*{
			\myvphantom
			h-\sigma_h\j
		}^{\text{convex}})
		{}+{}
		(\overbracket*{
			\myvphantom
			-f-\sigma_{-f,h}h
		}^{\text{convex}})
		{}+{}
		\bigl(\overbracket*{
			\myvphantom
			\tfrac{1+\gamma\sigma_{-f,h}-c}{\gamma}\sigma_h
			{}-{}
			2L_{\f}
		}^{=0}
		\bigr)\j
	\end{align*}
}%
altogether proving \eqref{eq:SD}.
Similarly, inequality \eqref{eq:L:geq} follows by observing that
\(
	0
{}\leq{}
	\D_{\h-2\f-\frac c\gamma h}
{}={}
	\D_{\h-\f}-\D_{\f}-\frac c\gamma\D
{}\leq{}
	\D_{\h-\f}-\frac c\gamma\D
\),
so that
\[
	\LL(x,x^-)
{}\overrel{\ref{thm:env:eq}}{}
	\varphi(\bar x)
	{}+{}
	\overbracket*{
		\D_{\h-\f}
		\vphantom{x^-}
	}^{\geq\frac c\gamma\D}(\bar x,x)
	{}+{}
	\overbracket*{
		\D_{\f}(\bar x,x^-)
	}^{\geq0}
	{}+{}
	\tfrac{c}{2\gamma}\D(x,x^-).
\]
The assertion of \cref{thm:SD:inf} once again follows from the same arguments as in the proof of \cref{thm:fsmooth:general}.

	\end{appendix}

	\ifsiam
		\bibliographystyle{siamplain}
	\else
		\phantomsection
		\addcontentsline{toc}{section}{References}%
		\bibliographystyle{plain}
	\fi
	\bibliography{TeX/Bibliography_abbr.bib}

\end{document}